\documentclass[12pt,reqno]{amsart}
\usepackage{geometry}
\usepackage{graphicx}				
\usepackage{indentfirst}
\usepackage{hyperref}
\usepackage{ulem}	
\usepackage{amsthm}
\usepackage{amsfonts,bm}
\usepackage{amsmath}
\allowdisplaybreaks[4]
\usepackage{amscd}
\usepackage{enumerate,footnote}
\usepackage{amssymb}	
\usepackage{verbatim}
\usepackage{cancel}
\usepackage{cite}
\usepackage{textcomp}								
\usepackage{color}
\usepackage{mathrsfs}
\usepackage{appendix}
\usepackage{enumitem}   
\setlist[itemize,2]{label=\textbullet} 

\newtheorem{example}{Example}[section]							
\newtheorem{theorem}{Theorem}[section]
\newtheorem{definition}{Definition}[section]
\newtheorem{proposition}{Proposition}[section]
\newtheorem{lemma}{Lemma}[section]
\newtheorem{remark}{Remark}[section]
\newtheorem{question}{Question}[section]
\newtheorem{problem}{Problem}[section]
\newtheorem{corollary}{Corollary}[section]

\def\a{\alpha}
\def\b{\beta}
\def\e{\varepsilon}

\def\loc{\text{loc}}

\def\F{\mathcal{F}}

\def\R{\mathbb R}
\def\P{{\mathcal P}}

\def\rd{\mathrm d}

\def\be{\begin{align}}
\def\ee{\end{align}}

\title{Construction of equilibrium states revisited}

\author{Changguang Dong}
\address{Chern Institute of Mathematics and LPMC, Nankai University, Tianjin 300071, China}
\email{dongchg@nankai.edu.cn}

\author{Qiujie Qiao}
\address{Chern Institute of Mathematics and LPMC, Nankai University, Tianjin 300071, China}
\email{qiujieqiao@mail.nankai.edu.cn}

\date{\today}



\begin{document}

\begin{abstract}
In \cite{PP-24}, Parmenter and Pollicott establish an abstract criterion that gives a geometric construction of 
equilibrium states for a class of partially hyperbolic systems. 
We refine their criterion to cover a much broader class of diffeomorphisms, which include certain diffeomorphisms with exponential mixing property (with respect to volume), Katok maps and ``almost Anosov'' diffeomorphisms.  As a special case, 
we obtain a  construction of equilibrium states for ergodic 
partially hyperbolic affine maps/flows on homogeneous spaces, without any restrictions on the orbit growth along center directions.
\end{abstract}

\maketitle


\section{Introduction}\label{introduction}

The study of equilibrium states attracts much attention in dynamical systems during past decades. The central topics in this area are the existence and uniqueness of equilibrium states, which are well understood only for a limited class of systems. 
The main purpose of this paper is to address the existence part for a certain class of partially hyperbolic diffeomorphisms. More precisely, we will give an explicit geometric construction of equilibrium states for these systems.


\subsection{Some history}

There are three main approaches to constructing equilibrium states: the first is based on Markov partitions, the second on the specification property, and the third on geometric methods. The first two belong to the framework of thermodynamic formalism, which borrows tools from statistical physics to describe the behavior of typical orbits.

\subsubsection{Markov partition}

 Markov partitions were introduced into dynamical systems by Adler-Weiss and Sinai in the late 1960s \cite{AW67,AW70,Sin68a,Sin68b}. 
A Markov partition decomposes the phase space into pieces such that the original dynamics can be encoded by a shift on symbolic sequences, thereby reducing the system to a simpler symbolic dynamics. 
Following this framework, 
 Bowen established the existence and uniqueness of equilibrium states for uniformly hyperbolic systems with H\"older continuous potentials \cite{Bow75a}. However, it becomes very delicate beyond uniform hyperbolicity, and until now there is not much progress in this direction.
 For  $C^{1+\alpha}$ surface diffeomorphisms with positive topological entropy, Sarig demonstrated that any such diffeomorphism had at most countably many ergodic invariant probability measures of maximal entropy \cite{Sa13}. 
 Analogous results held for higher-dimensional diffeomorphisms \cite{BO18} and for three-dimensional flows \cite{LS19}. 
 Recently, Buzzi, Crovisier, and Sarig showed that $C^\infty$ surface diffeomorphisms with positive topological entropy had finitely many ergodic measures of maximal entropy in general and exactly one in the topologically transitive case \cite{BCS22}.
For further information on symbolic dynamics for non-uniformly hyperbolic systems, see the survey \cite{Li21}.

\subsubsection{Specification property}

Specification was introduced by Bowen in \cite{Bow71}. 
Without using Markov partitions, 
Bowen provided an alternative proof of existence and uniqueness 
of equilibrium states for uniformly hyperbolic systems with  H\"older continuous potentials \cite{Bow75b}. 
This method relied on the specification and expansivity properties of the system, together with the Bowen property of potentials.  
This approach was later extended to flows by Franco \cite{Fr77}. 
More recently, Climenhaga and Thompson generalized these ideas to non-uniform settings, 
establishing existence and uniqueness of equilibrium states 
for homeomorphisms and flows under 
non-uniform versions of specification and expansivity \cite{CT-advance}. 
This framework was applied to various smooth systems with weak hyperbolicity, 
including Bonatti-Viana diffeomorphisms \cite{CFT18}, 
Ma\~n\'e diffeomorphisms \cite{CFT19}, 
geodesic flows on compact rank-one manifolds \cite{BCFT18}, 
and the Katok map \cite{Wa21}. 
For further recent developments in this direction, 
the reader is referred to the survey \cite{CT21}.
Subsequently, refinements of the Climenhaga-Thompson criterion were developed in \cite{PYY22, PYY25-2} 
and applied in  
\cite{PYY25} to establish results on equilibrium states for the Lorenz attractor. 

\subsubsection{Geometric construction}

Equilibrium states can also be constructed via geometric approaches in various partially hyperbolic settings. 
For partially hyperbolic diffeomorphisms with Lyapunov stability, Climenhaga, Pesin, and Zelerowicz \cite{CPZ-JMD} proved the existence and uniqueness of equilibrium states for continuous potentials 
satisfying both the u-Bowen and cs-Bowen properties. 
Their approach was based on pushing forward a reference measure on local unstable manifolds, which was constructed via a Carathéodory-type method. 
This result generalized earlier work of Spatzier and Visscher \cite{SV18} on isometric extensions of Anosov systems to a broader class of manifolds. Li and Wu \cite{LW23} further showed that when the system is topologically mixing, the unique equilibrium state has the Bernoulli property. 

Independently, Carrasco and Rodriguez Hertz developed another geometric method to establish existence and uniqueness of equilibrium states for certain H\"older continuous potentials on systems with center isometries, such as regular Anosov actions \cite{CR24}. 
Their approach provided refined control over transverse measures along invariant foliations. 
This additional structure played a key role in their subsequent work \cite{CR24-2}, which strengthened the Furstenberg-Marcus theorem by establishing unique ergodicity of the unstable foliation for mixing Anosov flows.


For $C^{1+\alpha}$ uniformly hyperbolic systems $f$, the equilibrium state corresponding to the geometric potential $\varphi=\log \left| \det(Df|E^u) \right|$ is precisely the unique SRB measure. 
A common construction of such measures relied on 
pushing forward the normalized volume supported on a piece of unstable manifold \cite{Sin68a, Ru76}. 
Pesin and Sinai later showed that, for partially hyperbolic attractors, 
any weak* limit of time averages of such pushforwards yields a Gibbs $u$-state \cite{PS82}. 
Recently, Parmenter and Pollicott introduced a modified version of this construction: by modifying the density at each step, they developed a new method for constructing equilibrium states in uniformly hyperbolic systems \cite{PP22-1}. 
Building on this idea, they established the existence of equilibrium states for all continuous potentials in partially hyperbolic systems that were Lyapunov stable or whose center-unstable bundle exhibited subexponential contraction \cite{PP22-2}. 
Subsequently, they extended their framework to Smale spaces and provided a new construction of equilibrium states for subshifts of finite type \cite{PP-24}.

\subsection{Motivations of our work}

In \cite{HHW17}, Hu, Hua and Wu introduced the notions of unstable metric entropy and unstable topological entropy defined on unstable manifolds, 
which characterized the structural complexity of orbits along local unstable manifolds. They also established a corresponding variational principle. 
Later, Hu, Wu, and Zhu extended this framework to 
unstable pressure for continuous potentials \cite{HWZ21} and introduced the notion of u-equilibrium states. 
By applying the upper semicontinuity property, 
they showed that a u-equilibrium state existed for any continuous potential. 
In subsequent work, Hu and Wu developed a notion of unstable entropy 
along invariant lamination $\mathcal{L}$ for general $C^{1+\alpha}$ ($\alpha > 0$) diffeomorphisms \cite{HW24}. 
Independently, Yang studied unstable metric entropy and 
introduced a new class of robustly transitive diffeomorphisms in \cite{Yang21}. 

Following these developments, 
we can analogously define the unstable metric entropy and the unstable topological pressure 
for general expanding foliations $\F$. 
Based on these notions, we introduce the notion of $\F$-equilibrium states; see Definition \ref{def,Fequilibriumstate}.
This naturally leads to the following question: 
\begin{question}\label{question 1}
Can one provide an explicit construction of $\F$-equilibrium states? 
\end{question}

Under an additional condition, 
by combining the theory of unstable entropy along expanding foliations with 
a geometric construction of equilibrium states, we establish the existence of $\F$-equilibrium states for general expanding foliations 
(Theorem \ref{existence F-es} and Corollary \ref{existence F-es-corollary}), thereby providing a partial answer to Question \ref{question 1}. 

The definition of unstable topological pressure involves taking the supremum over local unstable leaves. 
A natural question arises: 
\begin{question}\label{question 2}
Under what conditions is 
 the unstable topological pressure for an expanding foliation leaf-independent? 
 This means that for every local leaf $W^{\F}(x,\delta)$, 
 the unstable topological pressure yields the same value. 
\end{question}
In addition, we can ask the following question: 
\begin{question}\label{question 3}
Under what conditions does 
every local leaf $W^{\F}(x,\delta)$ admit the construction of an $\F$-equilibrium state? 
In these cases, what is the relation between $\F$-equilibrium states 
and classical equilibrium states? 
\end{question}
To address these issues, 
we develop an abstract criterion (Theorem \ref{F-pressure-2-corollary}) that provides an answer to both Question \ref{question 2} and Question \ref{question 3}. 
This criterion applies to a broad class of systems that admit 
an unstable or weakly unstable foliation, 
exhibit subexponential growth in transverse directions, 
and possess a strong form of mixing property; 
see the precise statements of conditions {\bf (C1)}-{\bf (C3)} in the next subsection.

\subsection{Main results}
Our main result demonstrates that, under these conditions, 
equilibrium states can be explicitly constructed by 
 pushing forward weighted volume measures supported on any local leaves.

For a diffeomorphism $f:M\to M$ on a closed Riemannian manifold $M$, 
we define the following conditions: 
 \begin{itemize}
\item[{\bf (C1)}]
 Either
\begin{itemize}
\item  $f$ is a $C^1$ partially hyperbolic diffeomorphism, 
where we consider the unstable foliation $W^u$ as the expanding foliation $\F$, 
and the center-stable distribution $E^{c s}=E^c\oplus E^s$ integrates to a foliaiton $W^{cs}$; or
\item $f$ is  a $C^{1+\alpha}$ diffeomorphism, where the unstable lamination $W^u$ integrates 
into an expanding foliation $\F$, 
and the stable lamination integrates into a foliation, which we also denote by $W^{cs}$. 
\end{itemize}

\item[{\bf (C2)}] 
For every $\epsilon>0$, 
there exists an increasing function $g^{\epsilon}:\mathbb{N}\to \mathbb{R}$ 
satisfying $$\lim\limits_{n \rightarrow \infty}n^{-1} \log g^\epsilon(n) = 0$$
such that 
for all $y\in M$ and $z \in W^{cs}_{\text{loc}}(y) $, we have 
\[
d^{cs}(y, z)\leq \epsilon g^{\epsilon}(n)^{-1} \Rightarrow  d^{cs}(f^{n}y, f^{n}z)\leq \epsilon. 
\]
\item[{\bf (C3)}]  
For each $x\in M$ and $\delta, \epsilon>0$, 
there exists a function $h_{x,\delta}^{\epsilon}: \mathbb{N} \to \mathbb{N}$ 
with $\lim\limits_{n\to\infty}h_{x,\delta}^{\epsilon}(n)n^{-1}=0$
such that for all $y\in M$, 
\begin{align*}
f^{h_{x,\delta}^{\epsilon}(n)}\left(\overline{W^{u}(x,\delta)}\right)\cap W^{cs}(y,\epsilon g^\epsilon(n)^{-1})\neq \varnothing, 
\end{align*}  
where $g^\epsilon$ is the function from condition {\bf (C2)}.
\end{itemize}

Our first result is the following abstract criterion, 
which is established under the conditions {\bf (C1)}-{\bf (C3)}. 

\begin{theorem}\label{F-pressure-2-corollary}
  Let $ M $ be a closed Riemannian manifold and 
  $ f: M \rightarrow M $ a diffeomorphism satisfying conditions {\bf (C1)}-{\bf (C3)}, 
and $\phi:M \to \mathbb{R}$ a continuous function. 
For any $x\in M$ and $\delta>0$, 
consider probability measures $ \left(\lambda_{n}\right)_{n=1}^{\infty} $ 
supported on $W^{u}(x,\delta)$ 
and absolutely continuous with respect to the induced volume $ \lambda_{W^{u}(x,\delta)} $ with densities
\begin{align*}
\frac{\rd \lambda_{n}}{\rd \lambda_{W^{u}(x,\delta)}}(y):=\frac{\exp \left(S_n(\phi-\Phi^{u})\left( y\right)\right)}{\int_{W^{u}(x,\delta)} \exp \left(S_n(\phi-\Phi^{u})\left(z\right)\right) \,\rd \lambda_{W^{u}(x,\delta)}(z)}, \quad \text { for } y \in W^{u}(x,\delta),
\end{align*}
where $\Phi^{u}(x):=-\log \left| \det\left( Df|E^{u}_x \right) \right|$. 
Define the time-averaged measures
\begin{align*}
\mu_{n}:=\frac{1}{n} \sum_{k=0}^{n-1} \left( f^k \right)_{*} \lambda_{n}, \quad n \geq 1. 
\end{align*}
Then each accumulation point of $\mu_{n}$ in the weak* topology is an equilibrium state for $\phi$. 
\end{theorem}

Since every equilibrium state is an $\F$-equilibrium state, 
the preceding theorem answers both Question \ref{question 2} and Question \ref{question 3} in this setting. 
Conditions {\bf (C1)}-{\bf (C3)} are satisfied for partially hyperbolic diffeomorphisms with Lyapunov stability.
A detailed comparison with related results will be given in a later subsection.

In the case of exponential mixing systems, we can refine the result as follows:

\begin{theorem}\label{exponential mixing lemma-theorem}
 Let $ M $ be a closed Riemannian manifold and 
  $ f: M \rightarrow M $ a diffeomorphism preserving a smooth measure $\mu$. 
  Suppose that 
  \begin{itemize}
\item $f$ satisfies conditions {\bf (C1)} and {\bf (C2)};
\item the continuous foliation $W^{cs}$  in {\bf (C1)} has $C^1$ leaves;
\item $f$ is exponential mixing with respect to the smooth measure $\mu$. 
  \end{itemize}
  Then the result of Theorem \ref{F-pressure-2-corollary} holds, namely, for any $x\in M$ and $\delta>0$, 
each accumulation point of the sequence $\mu_n$ is an equilibrium state for $\phi$.
\end{theorem}

Examples of exponential mixing maps include Anosov diffeomorphisms \cite{Bow75a}, 
 partially hyperbolic automorphisms of nilmanifolds \cite{GS14}, 
 mostly contracting systems \cite{Do00}, and 
 partially hyperbolic translations on homogeneous spaces \cite{KM96}.

Next, we present two examples of nonuniformly hyperbolic systems that 
satisfy all the assumptions of Theorem \ref{F-pressure-2-corollary}.

The first example is the Katok map $G_{\mathbb{T}^2}$, a smooth nonuniformly hyperbolic diffeomorphism of the $2$-torus $\mathbb{T}^2$. This map was originally constructed by Katok in \cite{Ka79} by locally perturbing a linear Anosov map near the origin to introduce nonuniform hyperbolicity. In the same work, Katok also proved that smooth systems with nonzero Lyapunov exponents exist on arbitrary compact surfaces without topological obstructions. Subsequently, Dolgopyat and Pesin \cite{DP02} extended this result to all compact manifolds of dimension $\geq 3$.

\begin{theorem}\label{Katok-existence F-es}
Let $G_{\mathbb{T}^2}$ be the Katok map on $\mathbb{T}^2$, and $\phi:\mathbb{T}^2 \to \mathbb{R}$ a continuous function.
Then Theorem \ref{F-pressure-2-corollary} applies with $f = G_{\mathbb{T}^2}$ and
$\Phi^{u}(x):=-\log \left| \det\left( DG_{\mathbb{T}^2}|E^u_x \right) \right|$. 
  For any $x\in M$ and $\delta>0$, 
each accumulation point of the sequence $\mu_n$ is an equilibrium state for $\phi$.
\end{theorem}

In particular, combining our results with works of Pesin–Senti–Zhang \cite{PSZ19} and Wang \cite{Wa21}, we obtain Corollary \ref{Katok-t1}: for the Katok map, any weak* limit of the time averages of pushforwards of normalized volume along unstable manifolds is a convex combination of the Dirac measure at the origin $\delta_0$ and the Lebesgue measure $m$ on $\mathbb{T}^2$.


Our second example comes from the class of ``almost Anosov'' diffeomorphisms,
introduced by Hu and Young in \cite{HY95}; see subsection \ref{Almost Anosov} for the definition.
Such diffeomorphisms are hyperbolic everywhere except at one point,
and they admit no SRB measure but an infinite SRB measure; see \cite{Hu00} for details.
Such phenomena are closely related to a key open problem in dynamics, 
namely Palis' Conjecture \cite{Pa05}, 
which predicts that for most dissipative diffeomorphisms 
there exist finitely many SRB measures whose basins 
cover a set of full Lebesgue measure.
This example illustrates how localized non-hyperbolicity
can obstruct the existence of SRB measures, thereby
contributing to our understanding of SRB measures in nonuniformly hyperbolic systems.

\begin{theorem}\label{Almost Anosov-existence F-es}
Let $f$ be an ``almost Anosov'' diffeomorphism on a compact manifold $M$, and $\phi: M \to \mathbb{R}$ a continuous function.
Then Theorem \ref{F-pressure-2-corollary} applies with 
$\Phi^{u}(x):=-\log \left| \det\left(D f|E^u_x \right) \right|$. 
  For any $x\in M$ and $\delta>0$, 
each accumulation point of the sequence $\mu_n$ is an equilibrium state for $\phi$.
\end{theorem}

\subsection{Comparison with Other Works}

We begin by recalling the notion of Lyapunov stability, introduced by F. Rodriguez Hertz, M. A. Rodriguez Hertz, and R. Ures in \cite{RRU07}. This notion is closely related to the condition of a topologically neutral center introduced by Bonatti and Zhang in \cite{BZ19}.

\begin{definition}
 A partially hyperbolic diffeomorphism $f$ has {\bf Lyapunov stability} 
 if for every $\epsilon>0$, there exists an $\epsilon_0>0$ such that 
 if $\gamma$ is a curve of length at most $\epsilon_0$, and if 
 $n \geq 0$ is such that $f^n \gamma$ is a cs-curve, then the length of 
 $f^n \gamma$ is at most $\epsilon$.
\end{definition}

Most previous works on geometric constructions of equilibrium states
focus on partially hyperbolic systems
where the center-stable direction exhibits Lyapunov stability \cite{CPZ-JMD, CR24, PP22-2}. 
For instance, by using geometric measure theory, Climenhaga, Pesin and Zelerowicz proved existence and uniqueness for certain potentials under the assumption of topological transitivity \cite{CPZ-JMD}. 
Later, Parmenter and Pollicott \cite{PP22-2} developed a more direct geometric construction, 
establishing the existence of equilibrium states for arbitrary continuous potentials. 
Additionally, for center isometries, based on the study of quasi-invariant measures along leaves of foliations, Carrasco and Rodriguez-Hertz established the existence and uniqueness of equilibrium states associated to some H\"older potentials.

Our work is motivated by the construction introduced by Parmenter and Pollicott \cite{PP22-2}. Our contribution has two innovations. 

Firstly, our criterion applies to both partially hyperbolic and nonuniformly hyperbolic systems. 
To our knowledge, applying geometric methods to construct equilibrium states in nonuniformly hyperbolic systems is new. 
We take the Katok map and almost Anosov diffeomorphisms as examples; for details, see Theorem \ref{Katok-existence F-es} and Theorem \ref{Almost Anosov-existence F-es}. 

Secondly, our criterion remains valid whenever the dynamics exhibit nontrivial growth in the transverse directions and satisfy a strong form of mixing, 
thus providing more flexibility. 

In particular, the systems considered in \cite{PP22-2} fit naturally into our framework as special cases, which are explained in detail as follows:
If a $C^{1+\alpha}$ partially hyperbolic system is topologically mixing and Lyapunov stable, 
then conditions {\bf (C1)}-{\bf (C3)} are satisfied: 
\begin{itemize}
\item Condition {\bf (C1)} follows directly from Theorem 2.4 in \cite{CPZ-JMD}.

\item Condition {\bf (C2)} is a direct consequence of Lyapunov stability: 
for any $\epsilon > 0$, one can choose $g^\epsilon \equiv \epsilon_0/\epsilon$, 
where $\epsilon_0$ is given by the definition of Lyapunov stability.

\item Condition {\bf (C3)} is proved in \cite[Lemma 3.7]{PP22-2} under the assumption of topological mixing. 
Furthermore, a weaker version of {\bf (C3)} holds under mere topological transitivity, 
as shown in \cite[Lemma 6.2]{CPZ-JMD}.
\end{itemize}
Moreover, Parmenter and Pollicott also
constructed equilibrium states by building conditional measures 
along the leaves of the $W^{cu}$ foliation \cite{PP22-2}, 
under the assumptions that $E^{cu}$ integrates to a foliation 
and that $f$ is uniformly contracting along $W^s$. 
Indeed, by applying our framework to the pair $(W^{cu}, W^s)$ instead of $(W^u, W^{cs})$, 
we also recover this result as a special case. 
In addition, our criterion remains valid even when 
$f$ exhibits expansion outside of $W^{cu}$, 
thereby extending the applicability to systems with a splitting of the form $TM = E^{cs} \oplus E^{cu}$.

We point out that 
while the method of pushing forward normalized volume along unstable manifolds 
has been successfully applied to uniformly hyperbolic systems \cite{Sin68a, Ru76} 
and partially hyperbolic systems \cite{PS82},
its extension to nonuniformly hyperbolic systems, 
even when the unstable lamination integrates to a continuous foliation $W^u$, 
remains unclear. 
In this paper, we address this gap by studying Katok maps and the ``almost Anosov'' diffeomorphisms introduced in \cite{HY95}. 
We prove that every limit of the time averages of pushforwards of volume along 
$W^u$ yields an equilibrium state for the geometric potential $-\log |\det(Df|E^u_x)|$. 
Furthermore, for the Katok map, any weak* limit of such time averages is a convex combination of the Dirac measure at the origin and the Lebesgue measure.  

\subsection{Further  problems}

For the Ma\~n\'e diffeomorphism \cite{Ma78} 
and the Bonatti-Viana example \cite{BV00}, 
our criterion does not directly apply 
due to the presence of exponential growth along the center-stable direction. 

\begin{problem}
Can one establish a modified criterion that is applicable to these diffeomorphisms? 
\end{problem}

 The Climenhaga-Thompson construction \cite{CT-advance} provides an approach that identifies a suitable collection of orbit segments, which we refer to
as the ``good'' set, whose topological pressure is strictly greater 
than that of the complement. 
When this good set satisfies appropriate specification and expansivity conditions, 
this method yields the existence and uniqueness of an equilibrium state for certain continuous potentials. 
This criterion applies to both the Ma\~n\'e diffeomorphism \cite{CFT19}
and the Bonatti-Viana example \cite{CFT18}. 
The Climenhaga-Thompson criterion 
may provide insights into how our criterion could be adapted to these more challenging cases. 

For the Katok map, as discussed earlier, fix $x\in\mathbb{T}^2$ and $\delta>0$. 
When pushing forward volume along the weak unstable leaf $W^u(x,\delta)$ under forward iteration, 
the resulting weak* limits are convex combinations of the Dirac measure $\delta_0$ and the Lebesgue measure $m$. 

\begin{problem}
Can one establish a criterion that 
determines how the coefficients in this convex combination 
depend on the parameters $x\in\mathbb{T}^2$ and $\delta>0$? 
In particular, does the Dirac measure $\delta_0$ necessarily appear in every such limit?
\end{problem}

\subsection*{Structure of this paper}
In \S\ref{preliminaries}, 
we review preliminary concepts from pressure theory, nonuniform hyperbolicity, 
and partial hyperbolicity.
In \S\ref{Unstable Entropy Along Expanding Foliations}, 
we introduce the notions and results related to unstable entropy along expanding foliations.
Section \ref{The construction of F-equilibrium states} is devoted to 
the construction of $\F$-equilibrium states 
for expanding foliations (see Theorem \ref{existence F-es}).
In Section \ref{Generating F-Equilibrium States from Each x}, 
we prove Theorem \ref{F-pressure-2-corollary}.  
The proofs of Theorems \ref{exponential mixing lemma-theorem}-\ref{Almost Anosov-existence F-es} are presented in \S \ref{Exponential Mixing Diffeomorphisms}-\ref{Almost Anosov} respectively.
Finally, in the Appendix, we include proofs of several auxiliary results.

\subsection*{Acknowledgement} This research was partially supported by National Key R\&D Program of China No. 2024YFA1015100. C. Dong was also grateful for the support by Nankai Zhide Foundation, ``the Fundamental Research Funds for the Central Universities" No. 100-63243066 and 100-63253093.
Q. Qiao was supported by Nankai Zhide Foundation.

\section{Preliminaries}\label{preliminaries}

\subsection{Entropy and pressure}

Let $(X,d)$ be a compact metric space, $f:X\to X$ a homeomorphism. 
We denote by $ \mathcal{M}(f,X)$ and $\mathcal{M}_e(f,X)\subset \mathcal{M}(f,X)$ 
the set of $f$-invariant and ergodic $f$-invariant Borel probability measures on $X$ respectively. 

For a partition $\xi$ of $M$, we denote by $\xi(x)$ the element of $\xi$ that 
contains $x$.
For two partitions $\xi$ and $\eta$, we write $\xi \geq \eta$ or $\eta\leq \xi$ 
to mean that  $\xi(x)\subset \eta(x)$ for all $x\in X$. 
A partition $\xi$ is called {\bf increasing} if $f^{-1}\xi \geq \xi$. 
For integers $m,n$ with $m\leq n$ and a partition $\eta$, we define the partition 
$\displaystyle \eta_{m}^{n}:=\bigvee_{i=m}^n f^{-i} \eta$. 

Fix a Borel probability measure $\mu$ on $M$. 
 A partition $\xi$ is called measurable if it is the limit of a refining sequence of finite partitions, 
 each of whose elements are measurable with respect to $\mu$. 
Furthermore, there exists conditional measures 
$ \left\{\mu_{x}^{\xi}\right\}_{x \in X} $ such that:
\begin{itemize}
\item[(1)] each $ \mu_{x}^{\xi}  $ is a probability measure on $ \xi(x) $;
\item[(2)] if $ \xi(x)=\xi(y) $, then $ \mu_{x}^{\xi}=\mu_{y}^{\xi} $;
\item[(3)] for every $ \psi \in L^{1}(X, \mu) $, 
$\int_{X} \psi \,\rd \mu=\int_{X} \int_{\xi(x)} \psi \,\rd \mu_{x}^{\xi} \,\rd \mu(x)$. 
\end{itemize}
Moreover, the conditional measures are unique in the following sense: 
if $ \left\{\mu_{x}^{\xi}\right\}_{x \in X} $ satisfies the above conditions, 
then $ \mu_{x}^{\xi}=\nu_{x}^{\xi} $ for $ \mu $-a.e. $ x $ (see \cite[Theorem 5.14]{EW11} for details).

We now describe some results on the entropy and pressure for probability measures. 

\begin{definition}
For a measurable partition $\xi$ and a probability measure $\mu$, we define 
the { information function} $I_\mu(\xi)$  
and the { entropy of the partition} $H_\mu(\xi)$ respectively by 
\begin{align*}
I_\mu(\xi)(x):=-\log \mu(\xi(x)), \quad H_\mu(\xi):=\int_X I_\mu(\xi)(x)\,\rd\mu(x).
\end{align*}
The conditional information function  $I_\mu(\xi|\eta)$ and 
the conditional entropy $H_\mu(\xi|\eta)$
with respect to a measurable partition $\eta$ are defined respectively by
\begin{align*}
I_\mu(\xi|\eta)(x):=-\log \mu_x^\eta(\xi(x)), \quad H_\mu(\xi|\eta):=\int_X I_\mu(\xi|\eta)(x)\,\rd\mu(x). 
\end{align*}
For each $\mu\in\mathcal{M}(f,M)$ and a measurable partition $\xi$, 
  the {\bf entropy} $h_{\mu}(f,\xi)$  
  and the {\bf conditional entropy} $h_\mu(f,\xi|\eta)$ with respect to a measurable partition $\eta$ are defined as 
  \[
  h_{\mu}(f,\xi):=\lim_{n \to \infty} \frac{1}{n} H_{\mu}\left( \xi_0^{n-1}\right), \quad   h_{\mu}(f,\xi|\eta):=\lim_{n \to \infty} \frac{1}{n} H_{\mu}\left( \xi_0^{n-1}|\eta\right). 
  \]
  The {\bf metric entropy} of $f$ is defined by 
  \[
  h_\mu(f):=\sup\{h_\mu(f,\xi) \mid \xi \text{ is a finite measurable partition with } H_\mu(\xi)<\infty \}. 
  \]
\end{definition}

Next, we recall some definitions and results about the topological pressure.

\begin{definition}\label{pressure}
Given $n\in \mathbb{N}$, $\epsilon>0$ and $x,y\in X$, we define $d_n$-metric and Bowen ball centered at $x$ of order $n$ and radius $\epsilon$ as 
\begin{align*}
  d_n(x,y)&:=\max_{0\leq i\leq n-1}d(f^ix,f^iy);\\
B(x,n,\epsilon)&:=\{y\in X: d_n(x,y)<\epsilon\}.
\end{align*}
A set $E\subset X$ is called {\bf($\boldsymbol{n},\boldsymbol{\epsilon}$)-spanning} if $X\subset \bigcup\limits_{x\in E}B(x,n,\epsilon)$. 
A set $E\subset X$ is called {\bf($\boldsymbol{n},\boldsymbol{\epsilon}$)-separated} if the $d_n$-distance 
between any two distinct points in $E$ is at least $\epsilon$.  
Given a continuous function $ \phi: X\to \mathbb{R}$, for any $x \in X $ and $n\in\mathbb{N}$, 
we define the statistical sum $\displaystyle S_{n} \phi(x):=\sum_{i=0}^{n-1} \phi\left(f^{i}(x)\right) $ and
  \begin{align*}
  N_{d}(\phi, n, \epsilon)&:=\sup \left\{\sum_{x \in E} e^{S_{n} \phi(x)} \mid E \subset X \text { is }(n, \epsilon) \text {-separated }\right\}; \\
  S_{d}(\phi, n, \epsilon)&:=\inf \left\{\sum_{x \in E} e^{S_{n} \phi(x)} \mid E \subset X \text { is }(n, \epsilon) \text {-spanning }\right\} .
  \end{align*}
  The {\bf topological pressure} of $ f$  with respect to  $\phi $ is 
  \[
    P_{\operatorname{top}}(f,\phi):=\lim _{\epsilon \rightarrow 0} \varlimsup_{n \rightarrow \infty} \frac{1}{n} \log S_{d}(f, \phi,n, \epsilon). 
    \]
In fact, when $\phi\equiv 0$, we replace $P_{\operatorname{top}}(f,\phi)$ with $h_{\operatorname{top}}(f)$ 
and refer to it as the {\bf topological entropy} of $f$. 
  \end{definition}

  Although the quantities $S_d(\phi,n,\epsilon)$ and $N_d(\phi,n,\epsilon)$ depend on the metric $d$, 
  the topological pressure is independent of the metric. 
The following variational principle establishes the relation between the metric entropy
and the topological pressure.

\begin{theorem}[\cite{PW} Variational Principle for pressure]\label{classical variational principle}
Let $ f: X \rightarrow X $ be a continuous map of a compact metric space 
$ X $, and $ \phi$ a continuous function. Then 
\begin{align*}
P_{\operatorname{top}}(f,\phi)&=\sup \left\{h_{\nu}(f)+\int \phi \, \rd \nu \mid \nu \in \mathcal{M}(f,X)\right\}\\
&=\sup \left\{h_{\nu}(f)+\int \phi \, \rd \nu \mid \nu \in \mathcal{M}_e(f,X)\right\}.
\end{align*}
\end{theorem}
Given a continuous function $\phi$, 
a $f$-invariant measure $\nu$ is called an {\bf equilibrium state} 
if it satisfies $P_{\text{top}}(f,\phi)=h_{\nu}(f)+\int \phi \, \rd \nu$. 
 In particular, the measure of maximal entropy is the equilibrium state for $\phi=0$. 

\subsection{Nonuniform hyperbolicity}

Nonuniformly hyperbolic system exhibit expanding and contracting directions, 
but the rates of expansion and contraction are not uniform across the entire space. 
Instead, these rates can vary from point to point, 
leading to a rich variety of behaviors and properties within the system.

Let $f$ be a $ C^{1+\alpha}$ diffeomorphism (with $\alpha>0$) on a 
closed Riemannian manifold $M$ and let 
$\mu$ be an invariant measure. 
A point $ x $ is called regular if there exist numbers 
$ \lambda_{1}(x)>\cdots>\lambda_{r(x)}(x) $ and 
a decomposition of the tangent space at $ x $ into $ T_{x} M=E_{1}(x) \oplus \cdots \oplus E_{r(x)}(x) $ 
such that for every nonzero tangent vector $ v \in E_{i}(x) $, 
\begin{align*}
\lim _{n \rightarrow \pm \infty} \frac{1}{n} \log \left\|D f_{x}^{n} v\right\| & =\lambda_{i}(x). 
\end{align*}
The numbers $ \lambda_{i}(x), i=1, \ldots, r(x) $, 
are called the Lyapunov exponents of $ f$ at $ x $, 
and the dimension $\operatorname{dim} E_{i}(x) $ is called the multiplicity of  $\lambda_{i}(x) $. 
According to the Multiplicative Ergodic Theorem \cite{Oseledec}, 
the set $\Gamma'$ of regular points has full measure for every $f$-invariant measure. 
Moreover, the functions $ x \mapsto r(x), \lambda_{i}(x)$  and 
$ \operatorname{dim} E_{i}(x) $ are measurable and invariant along orbits of $f$.

For every regular point $x\in \Gamma'$, we define 
$$
\lambda^+(x)=\min\{\lambda_i, \lambda_i>0\},\quad \lambda^-(x)=\max\{\lambda_i, \lambda_i<0\}
$$ 
and the following subspace of the tangent space: 
\begin{align*}
    E^s(x)  =\underset{\lambda_i(x)<0}{\bigoplus} E_i(x), \quad E^u(x)  =\underset{\lambda_i(x)>0}{\bigoplus} E_i(x), \quad E^c(x) & =E_{i_0}(x) \,\text{ with } \lambda_{i_0}(x)=0.
   \end{align*}
There exists a measurable function $ r(x)>0 $ such that for $ \mu $-almost every point $ x \in M $, 
the stable and unstable local manifolds at $ x $, defined as 
\begin{align*}
  W^{s}(x)&=\left\{y \in B(x, r(x)): \varlimsup\limits_{n \rightarrow +\infty} \frac{1}{n} \log d\left(f^{n} x, f^{n} y\right)<0\right\}, \\
  W^{u}(x)&=\left\{y \in B(x, r(x)): \varliminf\limits_{n \rightarrow -\infty} \frac{1}{n} \log d\left(f^{n} x, f^{n} y\right)>0\right\}. 
\end{align*}
They are immersed local manifolds. 
For any $0<r<r(x) $, the set $ W^{s}(x, r) \subset W^{s}(x) $ is the ball 
centered at $ x $
with respect to the induced distances on $ W^{s}(x) $, 
and similarly, we can define the set $ W^{u}(x, r) \subset   W^{u}(x) $. 

In general, nonuniformly hyperbolic systems exhibit laminations rather than foliations; 
for further details, refer to \cite{BP07, BP23}.

\subsection{Partially hyperbolicity}

The notion of partially hyperbolicity was
 introduced by Brin and Pesin \cite{BP74}. 
Partial hyperbolicity extends uniform hyperbolicity by permitting more flexibility 
in the behavior of the center directions.



\begin{definition}\label{Definition 2.1.}
  A diffeomorphism $f$ of a compact Riemannian manifold $M$ is called {\bf partially hyperbolic} if there are constants 
  \[
  0<\lambda_1 \leq \lambda_2<\gamma_1 \leq 1 \leq \gamma_2<\mu_1 \leq \mu_2,\quad C \geq 1
  \]
   and subspaces $E^s(x), E^c(x), E^u(x)$ for every $x \in M$ such that:
  \begin{itemize}
  \item[1.] $T_x M=E^s(x) \oplus E^c(x) \oplus E^u(x)$;
  
  \item[2.] the distributions $E^s, E^c$ and $E^u$ are invariant under the derivative $D f$, i.e. 
  \begin{align*}
  d f(x) E^v(x)=E^v(f(x)),\quad \text{for}\quad v=u, c, s;
  \end{align*}
  \item[3.] $C^{-1} \lambda_1^n\left\|v^s\right\| \leq\left\|d f^n(x) v^s\right\| \leq C \lambda_2^n\left\|v^s\right\|$ for any $v^s \in E^s(x)$ and $n\in\mathbb{N}$;
  
  \item[4.] $ C^{-1} \gamma_1^n\left\|v^c\right\| \leq\left\|d f^n(x) v^c\right\| \leq C \gamma_2^n\left\|v^c\right\|$ for any $v^c \in E^c(x)$ and $n\in\mathbb{N}$;
  
 \item[5.] $ C^{-1} \mu_1^n\left\|v^u\right\| \leq\left\|d f^n(x) v^u\right\| \leq C \mu_2^n\left\|v^u\right\|$ for any $v^u \in E^u(x)$ and $n\in\mathbb{N}$.
  \end{itemize}
\end{definition}

\begin{definition}
  For the partially hyperbolic diffeomorphism considered in Definition \ref{Definition 2.1.}, 
  for any $ x \in M$, the unstable manifold $W^u(x)$ is defined as
\begin{align*}
  W^u(x)=\{ y\in M \,|\, \exists\, c>0, \forall \, n\geq 0, d(f^{-n}(x), f^{-n}(y))<ce^{-\epsilon n} e^{-\gamma_2 n} \}. 
\end{align*}
  For a sufficiently small $ \delta>0 $, 
  we define $W^{u}(x,\delta)$ as an open ball in $W^{u}(x)$, 
  centered at $x$ with radius $\delta$, 
  and its metric is derived from the metric of the leaves of the unstable manifold.

\end{definition}

We have the following result about unstable manifolds.

\begin{theorem}[\cite{HPS77} Unstable Manifold Theorem]\label{KH-u}
For a $C^r$ ($r\geq 1$) partially hyperbolic diffeomorphism $f:M\to M$ 
on a compact closed Riemannian manifold, we have:
\begin{itemize}
\item[(1)] For any $x\in M$, the unstable manifold $W^u(x)$ is a $ C^{r} $ 
immersed submanifold of dimension $ \operatorname{dim}\left(E^{u}\right) $, 
and it is tangent to the tangent space $E^u(x)$ at $x$;
\item[(2)] For any $x,y\in M$, the unstable manifolds $W^u(x)$ and $W^u(y)$ 
are either equal or disjoint;
\item[(3)] For a sufficiently small $  \delta>0 $, 
$ W_{\delta}^{u}(x)$ depends continuously on $x$ and $f$ for the $C^r$-topology.
\end{itemize}
\end{theorem}

Similarly, the stable manifold can be defined in an analogous manner, 
and a corresponding Stable Manifold Theorem exists.
However, the center distribution is not always integrable 
 which leads to many difficulties. 
 For further details,  
 we refer the reader to \cite{Wil98,HPS77, Pe04} for more information.

\section{Unstable Entropy Along Expanding Foliations}\label{Unstable Entropy Along Expanding Foliations}

In this section, we introduce the definitions and results of unstable entropy along 
expanding foliations.

 Let $(M,d)$ be a closed Riemannian manifold and $f : M \to M$ a $C^1$ diffeomorphism.
 Given a continuous foliation $\mathcal{F}$ with $C^1$ leaves, 
 let $d^{\mathcal{F}}$ denote the metric  induced by the Riemannian structure 
 on leaves of $\F$ 
 and define 
 $$
 d^{\F}_n(y,z)=\max_{0\leq j\leq n-1} d^{\F}(f^jx, f^jy).
 $$
 Let $W^{\mathcal{F}}(x,\delta)$ denote the open ball centered at $x$ of
 radius $\delta$ with respect to the metric $d^{\F}$. 
 Sometimes, we write $W^{\mathcal{F}}_\delta(x)$ as shorthand to mean $W^{\mathcal{F}}(x,\delta)$. 
 Similarly, for a given positive integer $n$, 
 we denote by $W^{\mathcal{F}}(x,n,\delta)$ the open ball centered at $x$ of
 radius $\delta$ with respect to the metric $d_n^{\F}$. 

 A set $E\subset \overline{W^{\mathcal{F}}(x,\delta)}$ is called an {\bf ($\boldsymbol{n},\boldsymbol{\epsilon}$) $\F$-spanning set} of $ \overline{W^{\mathcal{F}}(x,\delta)}$ 
 if 
 $$
 \overline{W^{\mathcal{F}}(x,\delta)}\subset \bigcup_{z\in E} W^{\F}(z,n,\epsilon). 
 $$
 A set $E\subset \overline{W^{\mathcal{F}}(x,\delta)}$ is  
 {\bf ($\boldsymbol{n},\boldsymbol{\epsilon}$) $\F$-separated } 
 if pairwise $d_n^{\F}$-distances for any points in $E$ is at least $\epsilon$.

\begin{definition}\label{def: expanding foliation}
  A continuous foliation $\mathcal{F}$ is called an 
  {\bf expanding foliation} 
  if 
  \begin{itemize}
    \item[(1)] the foliation $\F$ has  $C^1$ leaves; 
  \item[(2)]  the foliation $\mathcal{F}$ is invariant under iterations of $f$;  
  \item[(3)]  for every $x\in M$ and $\delta>0$, there exists an integer $N>0$ such that for 
  all $n\geq N$ and $\epsilon>0$, if $E \subset f^n W^{\F}(x,\delta)$ is a subset with 
  $$
  f^n W^{\F}(x,\delta)\subset \bigcup_{z\in E} W^{\F}(z, \epsilon),
  $$ 
  then the preimage $f^{-n}E$ forms an $(n,\epsilon)$ $\F$-spanning set in $W^{\F}(x,\delta)$. 
  \end{itemize}
  \end{definition}

  \begin{remark}
The expanding foliation we discuss here is different from the one introduced by Yang in \cite{Yang21}. 
We do not require that the derivative $Df$ restricted to the tangent bundle of $\F$ is uniformly expanding.
  \end{remark}

  \begin{definition}
    A measurable partition $\xi$ of $M$ is said to be subordinate to $W^u_{loc} \cap \F$ with respect to 
    a measure $\mu\in \mathcal{M}(f,M)$, if for $\mu$-a.e. $x\in M$, 
    \[
    \xi(x)\subset W^u_{\text{loc}} \cap \F_{\text{loc}}(x)
    \]
    and $\xi(x)$ contains an open neighborhood of $x$ in $W^u_{\text{loc}} \cap \F_{\text{loc}}(x)$. 
    \end{definition}

In the above definition, we consider diffeomorphisms that are either $C^{1+\alpha}$
  or $C^1$ partially hyperbolic. The above definition applies to both cases with 
  their corresponding $W^u(x)$. 
We refer the reader to Section \ref{preliminaries} for these definitions.

Fix $\mu\in\mathcal{M}(f,M)$. 
We denote by $\mathcal{P}_\mu$ the set of finite measurable partitions of $M$. 
Let $\P^{\F}_\mu$ be the set of partitions subordinate to $W^u_{loc} \cap \F$ with respect to $\mu$. 
We will use  $\P_\mu$ and $\P^{\F}_\mu$ as shorthand for $\P$ and $\P^{\F}$, respectively, 
 when it doesn't cause confusion.

\begin{proposition}[\cite{HW24} Proposition 2.6]
Let $f:M\to M$ be a $C^{1+\alpha}$ diffeomorphism on a closed manifold $M$. 
For each $\mu\in\mathcal{M}(f,M)$, $ \mathcal{P}^{\F}_\mu$ is nonempty. 

Furthermore, for any $C^1$ partially hyperbolic diffeomorphism $f$, 
the set $ \mathcal{P}^{\F}_\mu$ is also nonempty, 
where $\F$ denotes the unstable foliation of $f$. 
\end{proposition}

For each $\mu\in \mathcal{M}(f,M)$, we define the \textbf{unstable entropy} of $\F$ by 
\begin{align*}
h_\mu^{\mathcal{F}}(f)=\sup_{\eta\in \mathcal{P}^{\F}}h_\mu(f|\eta), 
\end{align*}
where 
\begin{align*}
h_\mu(f|\eta)
=\sup_{\xi \in \mathcal{P}}h_\mu(f, \xi|\eta), \quad \text{and} \quad 
h_\mu(f, \xi|\eta)=\varlimsup_{n\to \infty}\frac{1}{n}H_\mu(\xi_0^{n-1}|\eta), 
\end{align*}

\begin{corollary}
For any invariant measure $\mu$, we have $h_\mu^{\mathcal{F}}(f)\leq h_\mu(f)$. 
\end{corollary} 

The following lemma gives the ergodic decomposition of unstable entropy.

\begin{lemma}[\cite{HHW17} (5.1), \cite{HW24} Corollary A.3]\label{decomposition}
Let $f:M\to M$ be a $C^{1+\alpha}$ diffeomorphism on a closed manifold $M$ 
and let $\mu\in\mathcal{M}(f,M)$. We denote by 
$\mu = \int \mu_{\tau(x)} \, \rd \mu (\tau(x))$ the ergodic decomposition of $\mu$. 
Given an expanding foliation $\F$, then 
\begin{align*}
h_{\mu}^{\F}(f)=\int h_{\mu_{\tau(x)}}^{\F}(f) \,\rd\mu(\tau(x)). 
\end{align*}
Consequently, for any continuous function $\phi: M\to \mathbb{R}$, we have 
\begin{align*}
h_\mu^{\F}(f)+\int_M \phi \,\rd\mu=\int\left(h_{\mu_{\tau(x)}}^{\F}(f)+\int_M \phi \,\rd\mu_{\tau(x)}\right) \,\rd\mu(\tau(x)). 
\end{align*}
Furthermore, for any $C^1$ partially hyperbolic diffeomorphism $f$, 
the above results hold when $\F$ denotes the unstable foliation of $f$. 
\end{lemma}

Given a continuous function $\phi: M\to \mathbb{R}$ and an expanding foliation $\F$, 
the {\bf partial topological pressure} with respect to $\phi$ is defined as 
\begin{align*}
P^{\mathcal{F}}_{\text{top}}(f, \phi, \overline{W^{\F}(x,\delta)})=\lim_{\epsilon \to 0}\varlimsup_{n\to \infty}\frac{1}{n}\log S^{\F}_d(\overline{W^{\mathcal{F}}(x,\delta)}, n, \epsilon), 
\end{align*}
where 
\begin{align*}
S^{\F}_d(\overline{W^{\mathcal{F}}(x,\delta)}, n, \epsilon):=\inf\left\{\sum_{y\in E} e^{S_n\phi(y)} \left| E \text{\ is an\ } (n, \epsilon) \right. \text{\ $\mathcal{F}$-spanning set of\ }\overline{W^{\mathcal{F}}(x,\delta)}\right\}. 
\end{align*}
When $\phi=0$, we refer to $P^{\mathcal{F}}_{\text{top}}(f, 0)$ as the 
{\bf partial topological entropy} of $f$, 
and we abbreviate the notions as $h^{\mathcal{F}}_{\text{top}}(f)$ and $h^{\mathcal{F}}_{\text{top}}(f, \overline{W^{\F}(x,\delta)})$. 

Naturally, we could define the partial topological pressure by using $(n,\epsilon)$ $\F$-separated sets. 
It is easy to see that 
\begin{align*}
P^{\mathcal{F}}_{\text{top}}(f, \phi)&:=\sup_{x\in M}P^{\mathcal{F}}(f, \phi, \overline{W^{\F}(x,\delta)}),
\end{align*}
where
\begin{align*}
P^{\mathcal{F}}_{\text{top}}(f, \phi, \overline{W^{\F}(x,\delta)}):=\lim_{\epsilon \to 0}\varlimsup_{n\to \infty}\frac{1}{n}\log N^{\F}_d(\overline{W^{\mathcal{F}}(x,\delta)}, n, \epsilon), 
\end{align*}
and
\begin{align*}
N^{\F}_d(\overline{W^{\mathcal{F}}(x,\delta)}, n, \epsilon):=\sup\left\{\sum_{y\in E} e^{S_n\phi(y)} \left| E \text{\ is an\ } (n, \epsilon) \right. \text{\ $\mathcal{F}$-separated set of\ }\overline{W^{\mathcal{F}}(x,\delta)}\right\}. 
\end{align*}
Similarly, we have 
\begin{align*}
P^{\mathcal{F}}_{\text{top}}(f, \phi, \overline{W^{\F}(x,\delta)}):=\lim_{\epsilon \to 0}\varlimsup_{n\to \infty}\frac{1}{n}\log S^{\F}_d(\overline{W^{\mathcal{F}}(x,\delta)}, n, \epsilon), 
\end{align*}
where
\begin{align*}
S^{\F}_d(\overline{W^{\mathcal{F}}(x,\delta)}, n, \epsilon):=\sup\left\{\sum_{y\in E} e^{S_n\phi(y)} \left| E \text{\ is an\ } (n, \epsilon) \right. \text{\ $\mathcal{F}$-spanning set of\ }\overline{W^{\mathcal{F}}(x,\delta)}\right\}. 
\end{align*}

From the definition, we deduce the following corollary.

\begin{corollary}\label{top pressure corollary}
For any expanding foliation $\F$, we have $P^{\mathcal{F}}_{\text{top}}(f, \phi)\leq P_{\text{top}}(f, \phi)$. 
\end{corollary}

It is worth mentioning that 
the definition of unstable metric entropy in our work is consistent with previous studies \cite{HW24}, 
while the definition of unstable topological entropy differs. 
This difference arises because earlier works considered invariant laminations, 
whereas we focus on invariant expanding foliations. 
Using the methods from \cite{HWZ21, HW24}, we obtain the following results.

\begin{lemma}\label{smalldelta}
The partial topological pressure is independent of $\delta$. 
More precisely, for any $\delta >0$, we have $P_{\operatorname{top}}^{\F}(f, \phi)=\sup\limits_{x\in M}P_{\operatorname{top}}^{\F}(f, \phi, \overline{W^u(x,\delta)})$. 
\end{lemma}

\begin{proof}
Fix $\delta >0$ and $\rho >0$, 
there exists $y \in M$ such that
\begin{align}\label{smalldelta-(1)}
\sup_{x\in M}P^{\F}_{\text{top}}(f, \phi, \overline{W^{\F}(x,\delta)})\leq P^{\F}_{\text{top}}(f, \phi, \overline{W^{\F}(y,\delta)})+\frac{\rho}{3}.
\end{align}
For this $\rho$, there exists $\epsilon_0>0$ such that 
\begin{align}\label{smalldelta-(2)}
\begin{aligned}
P^{\F}_{\text{top}}(f, \phi, \overline{W^{\F}(y,\delta)})&
=\lim_{\epsilon \to 0}\varlimsup_{n\to \infty}\frac{1}{n}\log S^{\F}_d(\overline{W^{\mathcal{F}}(y,\delta)}, n, \epsilon) \\ 
&\leq \varlimsup_{n\to \infty}\frac{1}{n}\log S^{\F}_d(\overline{W^{\mathcal{F}}(y,\delta)}, n, \epsilon_0)+\frac{\rho}{3}.
\end{aligned}
\end{align}
We can choose $\delta_1>0$ small enough such that $\delta_1<\delta$ and
\begin{align}\label{smalldelta-(3)}
P^{\F}_{\text{top}}(f, \phi) \geq\sup_{x\in M}P^{\F}_{\text{top}}(f, \phi, \overline{W^{\F}(x,\delta_1)})-\frac{\rho}{3}.
\end{align}
Then there exist $y_i \in \overline{W^{\F}(y,\delta)}, 1\leq i\leq N$ where $N$ only depends on $\delta$, $\delta_1$, and the Riemannian structure on $\overline{W^{\F}(y,\delta)}$, such that
\begin{align*}
\overline{W^{\F}(y,\delta)} \subset \bigcup_{i=1}^N\overline{W^{\F}(y_i, \delta_1)}.
\end{align*}
This means that for some $1\leq j\leq N$, 
\begin{align}\label{smalldelta-(4)}
S^{\F}_d(\overline{W^{\mathcal{F}}(y,\delta)}, n, \epsilon_0)\leq \sum_{i=1}^N S^{\F}_d(\overline{W^{\mathcal{F}}(y_j,\delta_1)}, n, \epsilon_0)\leq  N S^{\F}_d(\overline{W^{\mathcal{F}}(y_1,\delta_1)}, n, \epsilon_0). 
\end{align}
Combining \eqref{smalldelta-(1)}, \eqref{smalldelta-(2)}, \eqref{smalldelta-(3)} and \eqref{smalldelta-(4)}, we obtain 
\begin{align*}
\begin{aligned}
\sup_{x\in M}P^{\F}_{\text{top}}(f, \phi, \overline{W^{\F}(x,\delta)})&\leq P^{\F}_{\text{top}}(f, \phi, \overline{W^{\F}(y,\delta)})+\frac{\rho}{3} \\
&\leq\varlimsup_{n\to \infty}\frac{1}{n}\log S^{\F}_d(\overline{W^{\mathcal{F}}(y,\delta)}, n, \epsilon_0)+\frac{2\rho}{3} \\
&\leq \varlimsup_{n\to \infty}\frac{1}{n}\log \left(N S^{\F}_d(\overline{W^{\mathcal{F}}(y_j,\delta_1)}, n, \epsilon_0)\right)+\frac{2\rho}{3}  \\
&\leq \lim_{\epsilon \to 0} \varlimsup_{n\to \infty} \log \left(S^{\F}_d(\overline{W^{\mathcal{F}}(y_j,\delta_1)}, n, \epsilon)\right)+\frac{2\rho}{3}\\
& =P^{\F}_{\text{top}}(f, \phi, \overline{W^{\F}(y_j,\delta_1)})+\frac{2\rho}{3}\\
&\leq \sup_{x\in M}P^{\F}_{\text{top}}(f, \phi, \overline{W^{\F}(x,\delta_1)})+\frac{2\rho}{3}\leq P^{\F}_{\text{top}}(f, \phi)+\rho .
\end{aligned}
\end{align*}
Since $\rho>0$ is arbitrary, we have $\displaystyle \sup_{x\in M}P^{\F}_{\text{top}}(f, \phi, \overline{W^{\F}(x,\delta)})\leq P^{\F}_{\text{top}}(f, \phi)$.
 The reverse inequality is immediate from the definition, thus, we complete the proof. 
\end{proof}

\begin{proposition}\label{pro:hmuleqpressure}
  Let $f:M\to M$ be a $C^{1+\alpha}$ diffeomorphism on a closed manifold $M$ 
and let $\mu\in\mathcal{M}(f,M)$. 
Given an expanding foliation $\F$, we have 
$$
h_\mu^{\F}(f)+\int_M\phi \,\rd\mu \leq P_{\operatorname{top}}^{\F}(f, \phi).
$$
Moreover, 
for any $C^1$ partially hyperbolic diffeomorphism $f$, 
the above result holds when $\F$ denotes the unstable foliation of $f$ (see Theorem A in \cite{HWZ21}). 
\end{proposition}

The proof of the above result is almost identical to Proposition 4.6 in \cite{HWZ21}.  In the following we present a
 complete proof for the reader’s convenience.

We begin by introducing two auxiliary lemmas. 

\begin{lemma}[\cite{HWZ21} Lemma 4.5]\label{lem:inequality}
Suppose that $0 \leq p_1, \ldots, p_m \leq 1$, $s = p_1 + \cdots + p_m$ and $a_1, \ldots, a_m \in \mathbb{R}$. Then
\[
\sum_{i=1}^m p_i (a_i - \log p_i) \leq s \left( \log \sum_{i=1}^m e^{a_i} - \log s \right).
\]
\end{lemma}

\begin{lemma}[\cite{HHW17} Corollary 3.2, \cite{HW24} Corollary B.1]\label{entropy, Bowen ball}
Let $f:M\to M$ be a $C^{1+\alpha}$ diffeomorphism on a closed manifold $M$ 
and let $\mu\in\mathcal{M}_e(f,M)$. Then for any $\eta\in \mathcal{P}^{\F}$, 
\begin{align*}
h_\mu^{\F}(f)&=\lim_{\epsilon\to 0}\varliminf_{n\to\infty} -\frac{1}{n}\log \mu_x^\eta(W^{\F}(x,n,\epsilon))\\
&=\lim_{\epsilon\to 0}\varlimsup_{n\to\infty} -\frac{1}{n}\log \mu_x^\eta(W^{\F}(x,n,\epsilon)). 
\end{align*}
Furthermore, for any $C^1$ partially hyperbolic diffeomorphism $f$, 
the above results hold when $\F$ denotes the unstable foliation of $f$. 
\end{lemma}

We now ready to prove Proposition \ref{pro:hmuleqpressure}. 

\begin{proof}[Proof of Proposition \ref{pro:hmuleqpressure}]
By Lemma \ref{decomposition}, it suffices to prove this proposition for ergodic measures. 
Thus, without loss of generality, we assume that $\mu$ is ergodic.

Fix $\rho>0$. Take $\eta\in \P^{\F}$, 
according to Lemma \ref{entropy, Bowen ball}, 
for $\mu$ almost every point $x$, there exists $\epsilon_0(x)>0$, 
such that if $0<\epsilon<\epsilon_0(x)$, then 
\begin{align*}
\varliminf_{n \to \infty}-\frac{1}{n}\log\mu_x^\eta(W^{\F}(x,n,\epsilon))
\geq h_\mu^{\F}(f) -\frac{\rho}{2}.
\end{align*}
We choose a small $\epsilon'>0$ such that 
$\mu\left(B_{\epsilon'}\right)>1-\frac{\rho}{2}$, 
where $ B_{\epsilon'}:=\left\{x:\epsilon_0(x)>\epsilon'\right\}$.
For $\mu$-a.e. $x\in B_{\epsilon'}$, 
there exists $N(x)=N(x)>0$ such that for any $n\geq N(x)$, 
\begin{align}\label{e:estimate}
\mu_x^\eta(W^{\F}(x,n,\epsilon')) \leq e^{-n\left(h_\mu^{\F}(f)-\rho\right)}, \quad 
\frac{1}{n}(S_n\phi)(y)\geq \int_M \phi \,\rd\mu-\rho. 
\end{align}
For each $n\geq 1$, we define $E_n:=E_n(\epsilon')=\left\{x\in B_{\epsilon'}, N(x)\leq n\right\}$. 
Note that $\{E_n\}_{n\geq 1}$ forms an increasing sequence of sets. 
Obviously, there exists a positive integer $N_0$ such that $\mu\left(E_{N_0}\right)>\mu\left(B_{\epsilon'}\right)-\frac{\rho}{2}>1-\rho$. 
Fix $n>N_0$, there exists $x_0\in E_{n}$ such that
$$
\mu_{x_0}^\eta(E_{n})=\mu_{x_0}^\eta(E_{n}\cap \eta(x_0))>1-\rho.
$$
 Since the fact $y \in \eta(x_0)$ implies $\mu_y^\eta=\mu_{x_0}^\eta$, we obtain  
\begin{align}\label{e:entropyestimate}
\mu_{x_0}^\eta\left(W^{\F}(y,n, \epsilon')\right) \leq e^{-n(h_\mu^{\F}(f)-\rho)}, \quad
\forall y\in E_{n}\cap \eta(x_0).
\end{align}
We take $\delta>0$ such that $W^u(x_0, \delta)\supset \eta(x_0)$. 
Let $F_n$ be an $\left(n,\frac{\epsilon'}{2}\right)$ $\F$-spanning set of $\overline{W^{\F}(x, \delta)}\cap E_{n}$ 
satisfying $W^{\F}\left(z,n,\frac{\epsilon'}{2}\right)\cap \eta(x_0) \cap E_{n} \neq \varnothing$ for all $z\in F_n$ and 
\begin{align*}
\overline{W^{\F}(x_0,\delta)}\cap E_n \subset &\bigcup_{z\in F_n}W^{\F}\left(z,n,\frac{\epsilon'}{2}\right).
\end{align*}
For each $z\in F_n$, we choose $y(z)\in W^{\F}\left(z,n,\frac{\epsilon'}{2}\right)\cap \eta(x_0)\cap E_{n}$. 
We have 
\begin{align}\label{e:vpesti}
\begin{aligned}
1-\rho& <\mu_{x_0}^\eta(E_{n})\leq \mu_{x_0}^\eta(\overline{W^u(x_0,\delta)}\cap E_{n})\leq \mu_{x_0}^\eta\left(\bigcup_{z\in F_n}W^{\F}\left(z,n,\frac{\epsilon'}{2}\right)\right)\\
&\leq \sum_{z\in F_n}\mu_{x_0}^\eta\left(W^{\F}\left(z,n,\frac{\epsilon'}{2}\right)\right)\leq \sum_{z\in F_n}\mu_{x_0}^\eta\left(W^{\F}\left(y(z),n,\epsilon'\right)\right).
\end{aligned}
\end{align}
Combining \eqref{e:estimate} and \eqref{e:entropyestimate}, 
and applying Lemma \ref{lem:inequality} with constants $p_i=\mu_x^\eta(B_{n}^u(y(z),\e))$, $a_i=(S_n\phi)(y(z))$, 
we obtain
\begin{align*}
\begin{aligned}
&\sum_{z\in F_n}\mu_{x_0}^\eta\left(W^{\F}(y(z),n,\epsilon')\right)\left(n\left(\int_M \phi \,\rd\mu-\rho\right)+n(h_\mu^{\F}(f)-\rho)\right)\\
\leq &\sum_{z\in F_n}\mu_{x_0}^\eta\left(W^{\F}(y(z),n,\epsilon')\right)((S_n\phi)(y(z))-\log \mu_{x_0}^\eta\left(W^{\F}(y(z),n,\epsilon')\right))\\
\leq &\left(\sum_{z\in F_n}\mu_{x_0}^\eta\left(W^{\F}(y(z),n,\epsilon')\right)\right)  \cdot\\
&\quad\quad\left(\log \sum_{z\in F_n}\exp((S_n\phi)(y(z)))-\log \sum_{z\in F_n}\mu_{x_0}^\eta\left(W^{\F}(y(z),n,\epsilon')\right)\right).
\end{aligned}
\end{align*}
Combining \eqref{e:vpesti} with above inequality, we derive 
\begin{align}
\begin{aligned}\label{e:estim}
&n\left(\int_M \phi \,\rd\mu-\rho\right)+n(h_\mu^{\F}(f)-\rho)\\
\leq &\log \sum_{z\in F_n}\exp((S_n\phi)(y(z)))-\log \sum_{z\in F_n}\mu_{x_0}^\eta\left(W^{\F}(y(z),n,\epsilon')\right)\\
\leq &\log \sum_{z\in F_n}\exp((S_n\phi)(y(z)))-\log (1-\rho).
\end{aligned}
\end{align}
We define $\tau_{\epsilon'}:=\{|\phi(x)-\phi(y)|:d(x,y)\leq \epsilon'\}$. 
Therefore for each $z\in F_n$, we have $\exp((S_n\phi)(y(z)))\leq \exp((S_n\phi)(z)+n\tau_{\epsilon'})$. 
Dividing by $n$ and taking the limit superior with respect to $n$ on both sides of \eqref{e:estim}, 
by Lemma \ref{smalldelta}, we derive
\begin{align*}
\int_M \phi \,\rd\mu+h_\mu^{\F}(f)-2\rho & \leq \varlimsup_{n\to \infty}\frac{1}{n}\log \sum_{z\in F}\exp((S_n\phi)(y(z)))\\
&\leq \varlimsup_{n\to \infty}\frac{1}{n}\log \sum_{z\in F_n}\exp((S_n\phi)(z))+\tau_{\epsilon'}\leq P_{\text{top}}^{\F}(f, \phi)+\tau_{\epsilon'}. 
\end{align*}
Since $\rho>0$ is arbitrary and $\tau_\e \to 0$ as $\e \to 0$, we conclude 
$$
\int_M \phi \,\rd\mu+h_\mu^{\F}(f) \leq P_{\text{top}}^{\F}(f, \phi).
$$
Consequently, we finish the proof of this proposition. 
\end{proof}

In fact, the corresponding variational principle also holds. 
As for the proof, the argument from Theorem A in \cite{HWZ21} can be directly applied here. 
Since the conclusion is not required in this paper, we omit the detailed proof.

\begin{definition}\label{def,Fequilibriumstate}
A member $\mu\in \mathcal{M}(f,M)$ is called an {\bf $\F$-equilibrium state} for $\phi$ if 
\begin{align*}
P^{\F}_{\operatorname{top}}(f,\phi)=h_{\mu}^{\F}(f)+\int_M \phi \,\rd\mu. 
\end{align*}
\end{definition}

  \section{A construction of $\F$-equilibrium states}\label{The construction of F-equilibrium states}

Inspired by the work in \cite{PP22-1}, 
we formulate the following result, 
which is related to $\F$-equilibrium states.

  \begin{theorem}\label{existence F-es}
  Let $f: M \to M$ be a $C^{1+\alpha}$ diffeomorphism or a $C^1$ partially hyperbolic diffeomorphism on a closed Riemannian manifold $M$. 
  
  Given an expanding foliation $\mathcal{F}$ and a a continuous function  $\phi:M \to \mathbb{R}$, $x\in M$ and $\delta>0$, 
  consider probability measures $ \left(\lambda_{n}\right)_{n=1}^{\infty} $ 
  supported on $W^{\mathcal{F}}(x,\delta)$ 
  and absolutely continuous with respect to the induced volume $ \lambda_{W^{\mathcal{F}}(x,\delta)} $ with densities
  \begin{align*}
  \frac{\rd \lambda_{n}}{\rd \lambda_{W^{\mathcal{F}}(x,\delta)}}(y):=\frac{\exp \left(S_n(\phi-\Phi^{\F})\left( y\right)\right)}{\int_{W^{\mathcal{F}}(x,\delta)} \exp \left(S_n(\phi-\Phi^{\F})\left(z\right)\right) \,\rd \lambda_{W^{\mathcal{F}}(x,\delta)}(z)}, \quad \text { for } y \in W^{\mathcal{F}}(x,\delta) ,
  \end{align*}
where $\Phi^{\F}(x):=-\log \left| \det\left( Df|E^{\F}_x \right) \right|$ and 
 $E_x^{\F}$ is the bundle induced by the leaf of the foliation $\F$ at $x$. 
  Denote by $\mathcal{M}(x,\delta)$ the set of accumulation points in the weak* topology 
   of the averages 
  \begin{align*}
  \mu_{n}:=\frac{1}{n} \sum_{k=0}^{n-1} f_{*}^{k} \lambda_{n}, \quad n \geq 1. 
  \end{align*}
  Then for any $\mu\in \mathcal{M}(x,\delta)$, we have 
  $P^{\mathcal{F}}_{\operatorname{top}}(f, \phi, \overline{W^u(x,\delta)})\leq h^{\F}_\mu(f)+\int \phi\,\rd \mu$. 
\end{theorem}

By combining Lemma \ref{smalldelta}, Proposition \ref{pro:hmuleqpressure} and Theorem \ref{existence F-es}, 
  we can estabilish the existence of the $\F$-equilibrium state 
  by considering appropriate values of $x$. 
  This result provides a partial answer to Question \ref{question 1}. 
  
  \begin{corollary}\label{existence F-es-corollary}
    If there exists $x\in M$ such that 
    $P^{\mathcal{F}}_{\operatorname{top}}(f, \phi)= P^{\mathcal{F}}_{\operatorname{top}}(f, \phi, \overline{W^u(x_i,\delta)})$, 
   then every point in $\mathcal{M}(x,\delta)$ is an $\F$-equilibrium state. 
  \end{corollary}
  
An important consideration is whether the equality
$$ 
P^{\mathcal{F}}_{\text{top}}(f, \phi) = P^{\mathcal{F}}_{\text{top}}(f, \phi, \overline{W^{\mathcal{F}}(x,\delta)})
$$
holds uniformly for all $x\in M$ in certain special cases. 
This is closely related to Question \ref{question 1} and Question \ref{question 2} as previously mentioned. 
If this equality holds, the measures obtained from arbitrary choices of $x$ and $\delta$ in 
Theorem \ref{existence F-es} will automatically be $\mathcal{F}$-equilibrium states.

  However, achieving such uniformity requires the system to exhibit a strong type of mixing property, 
  which particularly challenging for the expanding foliation 
  whose dimension is lower than that of the unstable manifold. 
  In next section, we will discuss this situation in detail.

Before presenting the proof of this theorem, 
we will first discuss some properties of unstable topological pressure.

  Since $\phi: M \rightarrow \mathbb{R}$ is a continuous function on the compact manifold $M$, 
  the following lemma holds. 
  
  \begin{lemma}\label{potential lemma}
  Given $\tau>0$, there exists $\epsilon>0$ such that for all any $n\in\mathbb{N}$
   and $y, z \in M$ with $d_n\left(y, z\right)< \epsilon$, we have $\left|S_n\phi\left(y\right)-S_n\phi(z)\right| < n \tau$.
  \end{lemma}

  In fact, partial topological entropy can be regarded as the asymptotic rate of 
  orbit divergence along the expanding foliation. 
  For any $x\in M$ and $\delta>0$, 
  we use $ \lambda_{W^{\mathcal{F}}(x,\delta)} $ to denote the volume on $W^{\mathcal{F}}(x,\delta)$ induced by the Riemannina structure 
on leaves of $\F$.

  \begin{proposition}\label{F-pressure}
   For any $x\in M$, $\delta>0$ and continuous function $\phi: M\to\mathbb{R}$, we have 
  \begin{align*}
    P^{\mathcal{F}}_{\operatorname{top}}(f, \phi, \overline{W^{\F}(x,\delta)})
    &=\varlimsup_{n\to\infty}\frac{1}{n}\log \left(\int_{f^n W^{\F}_\delta(x)} e^{S_n\phi(f^{-n} y)}\,\rd \lambda_{f^n W^{\F}_\delta(x)}(y)\right)\\
    &= \varlimsup_{n \rightarrow+\infty} \frac{1}{n} \log \left(\int_{W^{\F}_\delta(x)}  e^{S_n \left(\phi-\Phi^{\F}\right)\left(y\right)} \,\rd \lambda_{W^{\F}_\delta(x)}(y)\right). 
  \end{align*}
  \end{proposition}
  
  \begin{proof}
For any positive integer $n$, applying the change of variables 
$$
f^n: W^{\F}_\delta(x)\to f^n\left( W^{\F}_\delta(x)\right), 
$$
we obtain 
\[
\int_{f^n W^{\F}_\delta(x)} e^{S_n\phi(f^{-n} y)}\,\rd \lambda_{f^n W^{\F}_\delta(x)}(y)
    = \int_{W^{\F}_\delta(x)}  e^{S_n \left(\phi-\Phi^{\F}\right)\left(y\right)} \,\rd \lambda_{W^{\F}_\delta(x)}(y). 
\]
Thus, it suffices to prove the first equality in this proposition. 
  
  Given $\tau>0$, 
  there exists $\epsilon=\epsilon(\delta)>0$ with $\epsilon<\delta$ such that the conclusion of 
   Lemma \ref{potential lemma} is satisfied. 
  
  Fix $n\in\mathbb{N}$, we choose a maximal set 
  $S=\left\{x_{1}, \cdots, x_{N}\right\}\subset f^{n} W^{\F}_\delta(x)$ such that
  \[
    W^{\F}\left(x_{i}, \frac{\epsilon}{4}\right) \cap W^{\F}\left(x_{j}, \frac{\epsilon}{4}\right)=\varnothing \quad \text{for} \quad i \neq j, \, 1\leq i,j\leq N
  \]
  and
  \[
    \bigcup_{i=1}^{N} W^{\F}\left(x_{i}, \frac{1}{4}\epsilon\right)\subset f^{n} \left(W^{\F}(x,\delta)\right), \quad f^{n} \left(\overline{W^{\F}(x,\delta) }\right)\subset \bigcup_{i=1}^{N} W^{\F}\left(x_{i}, \epsilon\right).
  \]
  By property (3) of Definition \ref{def: expanding foliation}, 
the set $f^{-n}S:=\left\{f^{-n}x_{1}, \cdots, f^{-n}x_{N}\right\}$ forms an $(n,\epsilon)$ $\F$-spanning set for 
  $ \overline{W^{\F}(x,\delta) }$.  
  Combining this with Lemma \ref{potential lemma}, we have 
  \begin{align*}
  S^{\F}_d(\overline{W^{\mathcal{F}}(x,\delta)}, n, \epsilon)  \leq& \sum_{i=1}^{N} e^{S_n \phi\left(f^{-n}x_{i}\right)} \\
   \leq& \sum_{i=1}^{N} \frac{e^{n\tau}}{\lambda_{f^{n}W^{\F}_\delta(x)}\left(W^{{\F}}\left(f^n x_{i}, \frac{\epsilon}{4}\right)\right)} \int\limits_{W^{\F}\left(f^n x_{i}, \frac{\epsilon}{4}\right)}  e^{S_n \phi\left(f^{-n}z\right)} \,\rd \lambda_{f^{n} W^{\F}_\delta(x)}(z)\\
  \leq&  \frac{1}{c_1} e^{n \tau} \int\limits_{f^{n} W^{\F}_\delta(x)} e^{S_n\phi\left(f^{-n} z\right)} \,\rd \lambda_{f^{n}W^{\F}_\delta(x)}(z), 
  \end{align*}
  where $c_1=\inf\limits_{z\in M} \lambda_{W^{\F}_\delta(z)}\left(W^{\F}\left(z, \frac{\epsilon}{4}\right)\right)>0$. 
  So it follows immediately that
  \begin{align}
  P^{\mathcal{F}}_{\text{top}}(f, \phi, \overline{W^{\F}(x,\delta)}) &=\lim _{\epsilon \rightarrow 0} \varlimsup_{n \rightarrow+\infty} \frac{1}{n} \log S^{\F}_d(\overline{W^{\mathcal{F}}(x,\delta)}, n, \epsilon)\nonumber \\
  &\leq \varlimsup_{n \rightarrow+\infty} \frac{1}{n} \log \int\limits_{f^n W^{\F}_\delta(x)} e^{S_n\phi\left(f^{-n} z\right)} \,\rd \lambda_{f^{n}W^{\F}_\delta(x)}(z)+\tau.\label{P-1}
  \end{align} 
  
  On the other hand, fix $n\in\mathbb{N}$, 
  we choose a maximal set $\left\{x_{1}, \cdots, x_{N}\right\}\subset f^{n} W^{\F}_\delta(x)$ such that
  $d^{\F}(x_i,x_j)>\epsilon$ whenever $i\neq j$. 
  Obviously,  points $y_i=f^{-n} x_i$ form an $(n,\epsilon)$ $\F$-separated set. 
  By Lemma \ref{potential lemma}, we have 
  \begin{align*}
  N^{\F}_d(\overline{W^{\mathcal{F}}(x,\delta)}, n, \epsilon) 
  &\geq \sum_{i=1}^N e^{S_n\phi(y_i)}=\sum_{i=1}^N e^{S_n\phi(f^{-n}x_i)}\\
  & \geq \sum_{i=1}^{N} \frac{e^{-n \tau}}{\lambda_{f^nW^{\F}_\delta(x)}\left(W^{\F}\left(x_{i}, \epsilon\right)\right)} \int_{W^{\F}\left(x_{i}, \epsilon\right)} e^{S_n \phi\left(f^{-n} z\right)} \,\rd \lambda_{f^n W^{\F}_\delta(x)}(z) \\
    &\geq \frac{e^{-n \tau}}{L} \int_{ f^n W^{\F}_\delta(x)} e^{S_n \phi\left(f^{-n} z\right)} \,\rd \lambda_{f^n W^{\F}_\delta(x)}(z) ,
    \end{align*}
  where $L=\sup\limits_{z\in M} \lambda_{W^{\F}_\delta(z)}\left(W^{\F}\left(z, \epsilon\right)\right)>0$. 
    Thus, we conclude 
  \begin{align}
  P^{\mathcal{F}}_{\text{top}}(f, \phi, \overline{W^{\F}(x,\delta)})&=\lim_{\epsilon\to 0}\varlimsup_{n \rightarrow+\infty} \frac{1}{n} \log N^{\F}_d(\overline{W^{\mathcal{F}}(x,\delta)}, n, \epsilon)\nonumber \\
  & \geq \varlimsup _{n \rightarrow+\infty} \frac{1}{n} \log \int_{f^{n} W^{\F}_\delta(x)} e^{\phi^{n}\left(f^{-n} z\right)} \,\rd \lambda_{f^n W^{\F}_\delta(x)}(z)-\tau.\label{P-2}
  \end{align}
 By letting $\tau\to 0$ in \eqref{P-1} and \eqref{P-2}, 
  we finish the proof of this proposition. 
  \end{proof}

  \begin{corollary}
    Given $\delta>0$ and a continuous function $\phi: M\to\mathbb{R}$,  we have 
  \begin{align*}
    P^{\mathcal{F}}_{\text{top}}(f, \phi)
    =\sup_{x\in M}\varlimsup_{n\to\infty}\frac{1}{n}\log \left(\int_{f^n W^{\F}_\delta(x)} e^{S_n\phi(f^{-n} y)}\,\rd \lambda_{f^n W^{\F}_\delta(x)}\right).
  \end{align*}
  \end{corollary}
  
For other studies of unstable volume growth in partially hyperbolic systems, 
we refer the reader to  \cite{HSX08,HHW17}.

We now present the following results.

\begin{proposition}[\cite{HW24} Proposition 5.1]\label{entropy property}
Suppose that $\mu=a\mu_1+(1-a)\mu_2$ where $\mu_1,\mu_2\in \mathcal{M}(M)$ and $0<a<1$. 
Let $\eta$ be a measurable partition with respect to $\mu$ and $\xi$ is a finite 
or countable measurable partitio with finite entropy for $\mu$, then 
$$
aH_{\mu_1}(\xi|\eta)+(1-a)H_{\mu_2}(\xi|\eta)\leq  H_{\mu}(\xi|\eta). 
$$

\end{proposition}

\begin{lemma}[\cite{HHW17} Lemma 2.6]\label{HHW17-Lemma 2.6}
  If $\alpha$ and $\gamma$ are measurable partitions, we have 
  \begin{align*}
  H_\mu(\alpha_0^{n-1}|\gamma)=H_\mu(\alpha|\gamma)+\sum_{i=1}^{n-1}
  H_\mu(\alpha | f^i(\alpha_0^{i-1}\vee \gamma)). 
  \end{align*}
  \end{lemma}
  
Next, we present an auxiliary lemma that was essentially contained in the proof of Theorem D in \cite{HHW17}. 
Since the original work did not explicitly state this result, 
we provide the complete proof below for the reader's convenience.

  \begin{lemma}[\cite{HHW17}]\label{conditional MM lemma}
    Given a probability measure $\nu$ and $n\in \mathbb{N}$, 
    we define $\mu_n := \dfrac{1}{n}\sum\limits_{i = 0}^{n - 1}f_*^i\nu$. 
    For any $0 < q < n$, $\alpha \in \mathcal{P}$, $\eta \in \mathcal{P}^{\F}$, we have
  \[
  \frac{1}{n}H_{\nu}(\alpha_0^{n - 1}|\eta)\leq \frac{3q}{n}\log \left(\operatorname{Card} \alpha\right)
   +\frac{1}{q}H_{\mu_n}(\alpha_0^{q - 1}|f\alpha^{\F}).
  \]
  where $\alpha^{\F}$ is a partition in $\mathcal{P}^{\F}$ whose elements are given by $\alpha^{\F}(x) = \alpha(x)\cap W^{\F}_{\text{loc}}(x)$. 
  \end{lemma}
  
  \begin{proof}
  For any integers $n,\,j$ with $n > q$ and $0\leq j\leq q-1$, we set $a(j) := [\frac{n - j}{q}]$, where $[a]$ represents the integer part of $a$. 
  Then we have
  \[
  \bigvee_{i = 0}^{n - 1}f^{-i}\alpha
  =\bigvee_{r = 0}^{a(j) - 1}f^{-(rq + j)}\alpha_0^{q - 1} \vee \bigvee_{t \in S_j}f^{-t}\alpha,
  \]
  where $S_j := \{0, 1, \cdots, j - 1\} \cup \{j + qa(j), \cdots, n - 1\}$.
  Thus, we derive 
  \begin{align*}
  H_{\nu}(\alpha_0^{n - 1}|\eta)
  &=H_{\nu}\left(\left.\bigvee_{r = 0}^{a(j) - 1}f^{-(rq + j)}\alpha_0^{q - 1} \vee
     \bigvee_{t \in S_j}f^{-t}\alpha\right|\eta\right)\\
  &\leq \sum_{t \in S_j}H_{\nu}(f^{-t}\alpha|\eta)
   +H_{\nu}\left(\left.\bigvee_{r = 0}^{a(j) - 1}f^{-rq - j}\alpha_0^{q - 1}\right|\eta\right).
  \end{align*}
  By using Lemma \ref{HHW17-Lemma 2.6}, we obtain
  \begin{align*}
  &H_{\nu}\left(\left.\bigvee_{r = 0}^{a(j) - 1}f^{-rq - j}\alpha_0^{q - 1}\right|\eta\right)=
  H_{f^j_*\nu}\left(\left.\bigvee_{r=0}^{a(j)-1}f^{-rq}\alpha_0^{q-1}\right|f^j\eta\right)\\
  = &H_{f^j_*\nu}(\alpha_0^{q-1}|f^j\eta)
  + \sum_{r=1}^{a(j)-1}H_{f^{rq+j}_*\nu}\Big(\alpha_0^{q-1}
   \Big|f^{rq}\big(\bigvee_{i=0}^{r-1}f^{-iq}\alpha_0^{q-1}\vee f^j\eta\big) \Big)\\
  \le&  H_{f^j_*\nu}(\alpha_0^{q-1}|f^j\eta)
  + \sum_{r=1}^{a(j)-1}H_{f^{rq+j}_*\nu}(\alpha_0^{q-1}|f\a^{\F}).
  \end{align*}
This implies that 
  \begin{align*}
  H_{\nu}(\alpha_0^{n - 1}|\eta)\leq \sum_{t \in S_j}H_{\nu}(f^{-t}\alpha|\eta)
  + H_{f^j_*\nu}\left(\left.\alpha_0^{q - 1}\right|f^j\eta\right)
  + \sum_{r = 1}^{a(j) - 1}H_{f^{rq + j}_*\nu}(\alpha_0^{q - 1}|f\alpha^{\F}).
  \end{align*}
  Obviously, the cardinality of $ S_j$ is no more than $2q$.  
  Summing the above inequality over $j$ from $0$ to $q - 1$, 
  by Proposition \ref{entropy property}, we have
  \begin{align*}
  qH_{\nu}(\alpha_0^{n - 1}|\eta)\leq &\sum_{j = 0}^{q - 1} \sum_{t \in S_j}H_{\nu}(f^{-t}\alpha|\eta)
     +\sum_{j = 0}^{q - 1}H_{f^j_*\nu}(\alpha_0^{q - 1}|f^j\eta)
     +\sum_{i = 0}^{n - 1}H_{f^i\nu}(\alpha_0^{q - 1}|f\alpha^{\F})\\
  \leq &2q^2\log \left(\operatorname{Card} \alpha\right)
  +q \log \left(\operatorname{Card} \alpha_0^{q - 1}\right)
   +nH_{\mu_n}(\alpha_0^{q - 1}|f\alpha^{\F})\\
   \leq &  3q^2\log \left(\operatorname{Card} \alpha\right)
   +nH_{\mu_n}(\alpha_0^{q - 1}|f\alpha^{\F}).
  \end{align*} 
 This finishes the proof. 
  \end{proof}
  
 We refer the reader to \cite{Gl03, EW11} for detailed background on martingale theory in dynamical systems, and a concise overview can be found in \cite{HHW17}. The following result will be useful for our purposes:


\begin{lemma}[\cite{Gl03} Theorem 14.28]\label{martingale lemma}
Let \( \alpha \in \mathcal{P} \) and \( \{\zeta_n\} \) be a sequence of increasing measurable partitions with \( \zeta_n \nearrow \zeta \). Then for any Borel probability measures $\mu$, 
 \(\displaystyle \lim_{n \to \infty} H_\mu(\alpha|\zeta_n) = H_\mu(\alpha|\zeta) \). 
\end{lemma}

Note that for each partition $\a\in \P$, the partition $\zeta$
given by $\zeta(x)=\a(x)\cap W^\F_\loc(x)$ for any $x\in M$
is an element in $\P^\F$.

  \begin{proposition}\label{muntomu}
Let $\mu\in \mathcal{M}(M)$.
For any $\a,\b\in \P$ and $\eta\in \P^{\F}$ with $\mu(\partial \a)=0$
and $\mu(\partial \b)=0$,
the map $\nu \mapsto H_\nu(\alpha|\b^{\F})$ from $\mathcal{M}(M)$ to
${\mathbb R}^+\cup\{0\}$ is upper semi-continuous at $\mu$,
i.e.
$$
\limsup_{\nu\to \mu}H_{\nu}(\alpha|\b^{\F})\leq H_\mu(\alpha|\b^{\F}).
$$
  \end{proposition}
  
  \begin{proof}
Since $\mu(\partial \b)=0$,
we can take a sequence of partitions
$\{\beta_n\} \subset \mathcal{P}$ such that $\beta_1<\beta_2< \cdots$
and $\beta_n\nearrow \b^{\F}$,
and moreover, $\mu(\partial \beta_n)=0$ for $n=1,2, \cdots$.

Since $\mu(\partial \alpha)=0=\mu(\partial \beta_n)$, and
for any invariant measure $\nu$,
\[
H_\nu(\alpha|\beta_n)
=-\sum_{A_i\in \alpha,B_j\in \beta_n}\nu(A_i \cap B_j)
 \log \frac{\nu(A_i \cap B_j)}{\nu(B_j)},
\]
we have $\displaystyle \lim_{\nu\to \mu}H_{\nu}(\alpha|\beta_n)= H_\mu(\alpha|\beta_n)$
for any $n\in \mathbb{N}$.
By Lemma \ref{martingale lemma}, we obtain 
$$ H_\nu(\alpha|\eta)=\lim_{n\to \infty}H_\nu(\alpha|\beta_n).$$
So for any $\epsilon>0$, there exists $N\in \mathbb{N}$ such that
$H_\mu(\alpha|\beta_N)\leq H_\mu(\alpha|\b^{\F})+\epsilon.$
One has
$$
\limsup_{\nu\to \mu}H_{\nu}(\alpha|\b^{\F})
\leq \limsup_{\nu\to \mu}H_{\nu}(\alpha|\beta_N)
=H_\mu(\alpha|\beta_N)
\leq H_\mu(\alpha|\b^{\F})+\epsilon.
$$
Since $\epsilon>0$ is arbitrary, we get the inequality.
  \end{proof}

  Next, we present the proof of Theorem \ref{existence F-es}.
  This proof is adapted from \cite[Theorem 1.2]{PP22-1}, 
  and relies on the growth rate result from Proposition \ref{F-pressure}. 
  
  \begin{proof}[Proof of Theorem \ref{existence F-es}]
  
Without loss of generality, we may assume that $\{\mu_n\}_n$ converges to $\mu$ in the weak* topology; 
otherwise, we can extract a convergent subsequence.

Given $\epsilon>0$, there exists $\tau>0$ such that the conclusion of Lemma \ref{potential lemma} holds.
  We choose a finite measurable partition 
  $\beta$ for $M$ satisfying 
  \begin{align*}
  \sup\left\{\operatorname{diam}\left(P_{i}\right): \forall\, P_i\in \beta \right\}<\epsilon. 
  \end{align*}
  This implies that for any $n\in\mathbb{N}$, 
  $A \in \bigvee\limits_{i=0}^{n-1} f^{-i} \beta $ and 
  $x, y \in A $, we have 
  \begin{align}\label{(4.1)}
  \left|S_n\phi\left( x\right)-S_n\phi\left( y\right)\right| \leq n \tau.
  \end{align}

By adapting the proof of \cite[Theorem 1.2]{PP22-1}, we establish the following lemma.

  \begin{lemma}\label{Lemma-negative}
   For sufficiently small $\epsilon$, for any $\displaystyle A \in \bigvee_{h=0}^{n-1} f^{-h} \beta$ and $C\in \eta$, 
   we have 
   \begin{align*}
   \lambda_{f^{n}W^{\mathcal{F}}(x,\delta)}\left(f^{n}(A\cap C)\right)\leq \lambda_{f^{n}W^{\mathcal{F}}(x,\delta)}\left(f^{n}(A)\right)<1.
   \end{align*}
  \end{lemma}
  
  \begin{proof}
  Fix $\epsilon>0$. By the continuity of $f$, 
  for any sufficiently small $\epsilon'>0$, $\displaystyle A \in \bigvee_{h=0}^{n-1} f^{-h} \beta$ and $p\in M$, 
   if 
   \begin{align}\label{Lemma-negative-(1)}
   f^{n-1} A \cap f^{n-1} W^{\mathcal{F}}(x,\delta) \subset W^{\mathcal{F}}(f^{-1}p,\epsilon'), 
   \end{align}
  then 
  $$
  f^n A \cap f^n W^{\mathcal{F}}(x,\delta) \subset W^{\mathcal{F}}(p,\epsilon). 
$$  

We now prove the existence of $p$ in \eqref{Lemma-negative-(1)}. 
We proceed by contradiction and assume that 
there does not exist $p\in f^{n} W^{\mathcal{F}}(x,\delta)$ such that 
$f^{n-1} A \cap f^{n-1} W^{\mathcal{F}}(x,\delta) \subset$ $W^{\mathcal{F}}(f^{-1}p,\epsilon')$. 
Then there exist 
  $z_1, z_2 \in f^{n-1} A\cap f^{n-1} W^{\mathcal{F}}(x,\delta)$ satisfying 
  $W^{\F}\left(z_1, \frac{\epsilon'}{2}\right) \cap W^{\F}\left(z_2, \frac{\epsilon'}{2}\right)=\varnothing$, 
  which implies $d^{\F}(z_1,z_2)> \epsilon'$. Therefore, we conclude 
   \[
  \max\limits_{0\leq i<n} d^{\F}\left(f^i\left(f^{-(n-1)}z_1 \right), f^i\left(f^{-(n-1)}z_2 \right)\right)>\epsilon'. 
  \] 
However, this  contradicts the fact that $f^{-(n-1)} z_1, f^{-(n-1)} z_2 \in A$ and $A \in B(z,n, \epsilon')$ for some $z\in M$. 
Therefore, the point $p$ must exist, and consequently, 
\[
f^{n} A \cap f^{n} W^{\mathcal{F}}(x,\delta) \subset W^{\mathcal{F}}(p,\epsilon).
\]
 By choosing $\epsilon$ sufficiently small, 
  we obtain 
  $\lambda_{f^{n} W^{\mathcal{F}}(x,\delta)}\left(f^{n}(A)\right)<1$. 
Thus, we finish the proof of this lemma. 
 \end{proof}

  For each $\displaystyle A \in \bigvee_{h=0}^{n-1} f^{-h} \beta$ and $C\in \eta$ with $A\cap C\neq \varnothing$, 
  we fix a point $x_{A,C} \in A\cap C$  and define 
  \[
  K_{n,A,C}:=\int_{f^{n}\left(A \,\cap \,C\,\cap\, W^{\mathcal{F}}(x,\delta)\right)} \exp \left(S_n\phi\left(f^{-n} y\right)\right) \rd \lambda_{f^{n}W^{\mathcal{F}}(x,\delta)}(y). 
  \]
    By Lemma \ref{Lemma-negative} and \eqref{(4.1)}, we derive
  \begin{align}
  \log K_{n,A, C}  \leq S_n \phi\left(x_{A, C}\right)+n \tau+\log \lambda_{f^n W^{\mathcal{F}}(x,\delta)}\left(f^{n}(A\cap C)\right) \leq S_n \phi\left( x_{A, C}\right)+n \tau .\label{(4.3)}
  \end{align}
  For each $0 \leq j \leq n-1$, we have
  \begin{align*}
  \int_{f^{j} W^{\mathcal{F}}(x,\delta)} \phi(y) \, \rd\left(f_{*}^{j} \lambda_{n}\right) (y)&=\int_{ W^{\mathcal{F}}(x,\delta)} \phi(f^jy) \, \rd\left( \lambda_{n}\right)(y)\\
  &=\frac{1}{Z_{n}^{\phi}} \int_{ W^{\mathcal{F}}(x,\delta)} \phi(f^jy) \, \rd\left( \lambda_{W^{\mathcal{F}}(x,\delta)}\right)(y)\\
  &=\frac{1}{Z_{n}^{\phi}} \int_{f^{n}W^{\mathcal{F}}(x,\delta)} \phi\left(f^{-(n-j)} y\right) e^{S_n\phi\left(f^{-n} y\right)}  \rd \lambda_{f^n W^{\mathcal{F}}(x,\delta)}(y),
  \end{align*}
  where 
  \begin{align*}
  Z_{n}^{\phi}&:=\int_{W^{\mathcal{F}}(x,\delta)} \exp \left(S_n(\phi-\Phi^{\F})\left(z\right)\right) \,\rd \lambda_{W^{\mathcal{F}}(x,\delta)}(z)\\
  &=\int_{f^{n}W^{\mathcal{F}}(x,\delta)} \exp \left(S_n\phi\left(f^{-n} y\right)\right) \rd \lambda_{f^n W^{\mathcal{F}}(x,\delta)}(y). 
  \end{align*}
   Hence, by \eqref{(4.1)} and the definition of $K_{n,A,C}$, we obtain 
  \begin{align}
  &\int_{M} \phi(y) \,\rd \mu_{n}(y) =\frac{1}{n}\sum_{j=0}^{n-1} \int_{f^{j} W^{\mathcal{F}}(x,\delta)} \phi(y) \,\rd (f^j_*\lambda_n)(y)\nonumber\\
  = & \frac{1}{nZ_{n}^{\phi}}\sum_{j=0}^{n-1}  \int_{f^{n}W^{\mathcal{F}}(x,\delta)} \phi\left(f^{-(n-j)} y\right) e^{S_n\phi\left(f^{-n} y\right)}  \rd \lambda_{f^n W^{\mathcal{F}}(x,\delta)}(y)\nonumber\\
 \geq & \frac{1}{n Z_{n}^{\phi}} 
  \sum_{A \in \bigvee\limits_{h=0}^{n-1} f^{-h} \beta} \sum_{C\in \eta} \sum_{j=0}^{n-1} \left(\phi(f^j (x_{A,C}))-\tau\right)   \int\limits_{f^{n}\left(A \cap C\cap W^{\mathcal{F}}(x,\delta)\right)} e^{S_n\phi\left(f^{-n} y\right)} \rd \lambda_{f^n W^{\mathcal{F}}(x,\delta)}(y)\nonumber\\
  =&\frac{1}{n Z_{n}^{\phi}} \sum_{A \in \bigvee\limits_{h=0}^{n-1} f^{-h} \beta}\sum_{C\in \eta} \left(S_n\phi\left(x_{A, C}\right)-n \tau\right) K_{n,A,C}. \label{(4.2)}
  \end{align}
We emphasize that although the partition $\eta$ is uncountable, 
for any local leaf $W^{\mathcal{F}}(x,\delta)$, 
the collection $\left\{C \in \eta : C \text{ contains an open subset of } W^{\mathcal{F}}(x,\delta)  \right\}$ 
is at most countable. 
Therefore, the summation $\displaystyle \sum_{C\in\eta}$ is well-defined 
since it reduces to a countable sum over 
those partition elements that intersect $W^{\mathcal{F}}(x,\delta)$. 

Noting that $\lambda_n(A\cap C)=\dfrac{K_{n,A,C}}{Z_n^\phi}$ and 
  $\sum\limits_{A \in \bigvee\limits_{h=0}^{n-1} f^{-h} \beta} \sum\limits_{C\in \eta}K_{n,A,C}=Z_{n}^{\phi}$, we have 
  \begin{align}
  H_{\lambda_{n}}\left(\left. \bigvee\limits_{h=0}^{n-1} f^{-h} \beta \right| \eta\right) & 
  =-\sum_{A \in \bigvee\limits_{h=0}^{n-1} f^{-h} \beta}\sum_{C\in \eta} \lambda_{n}(A\cap C) \log \frac{\lambda_{n}(A\cap C)}{\lambda_n(C)}\nonumber\\
  &\geq -\sum_{A \in \bigvee\limits_{h=0}^{n-1} f^{-h} \beta}\sum_{C\in \eta} \lambda_{n}(A\cap C) \log \lambda_{n}(A\cap C)\nonumber\\  
  & =-\sum_{A \in \bigvee\limits_{h=0}^{n-1} f^{-h} \beta} \sum_{C\in \eta} \frac{K_{n, A, C}}{Z_{n}^{\phi}} \log \frac{K_{n, A, C}}{Z_{n}^{\phi}}, \nonumber \\
  & =\log Z_{n}^{\phi}-\sum_{A \in \bigvee\limits_{h=0}^{n-1} f^{-h} \beta} \sum_{C\in \eta} \frac{K_{n, A, C}}{Z_{n}^{\phi}} \log K_{n, A, C}\label{(4.4)}. 
  \end{align}
 Applying inequalities \eqref{(4.3)} and \eqref{(4.4)}, we obtain 
  \begin{align}\label{(4.5)}
  H_{\lambda_{n}}\left(\left. \bigvee_{h=0}^{n-1} f^{-h} \beta \right| \eta\right) \geq 
  \log Z_{n}^{\phi}-\sum_{A \in \bigvee_{h=0}^{n-1} f^{-h} \beta} \sum_{C\in \eta}\frac{K_{n, A, C}}{Z_{n}^{\phi}}\left(S_n\phi\left( x_{A, C}\right)+n \tau\right). 
  \end{align}
  Combining \eqref{(4.2)}, \eqref{(4.5)} and $\sum\limits_{A \in \bigvee\limits_{h=0}^{n-1} f^{-h} \beta} \sum\limits_{C\in \eta}K_{n,A,C}=Z_{n}^{\phi}$, we conclude 
  \begin{align*}
  & H_{\lambda_{n}}\left(\left. \bigvee_{h=0}^{n-1} f^{-h} \beta \right| \eta\right)+n \int_{M} \phi \,\rd \mu_{n} \nonumber\\
  & \geq \log Z_{n}^{\phi}-\sum_{A \in \bigvee\limits_{h=0}^{n-1} f^{-h} \beta} \sum_{C\in \eta} \frac{K_{n, A, C}}{Z_{n}^{\phi}}\left(S_n \phi\left( x_{A, C}\right)+n \tau\right) \\
  &\quad +\frac{1}{Z_{n}^{\phi}} \sum_{A \in \bigvee\limits_{h=0}^{n-1} f^{-h} \beta}\sum_{C\in \eta} \left(S_n \phi\left( x_{A, C}\right)-n \tau\right) K_{n, A, C} \\
  & \geq \log Z_{n}^{\phi}-2 n \tau .\nonumber
  \end{align*}
  By Lemma \ref{conditional MM lemma}, for any integers $0<q<n$, we have 
  \begin{align*}
  \frac{1}{n} \log Z_{n}^{\phi}- \int_{M} \phi \,\rd \mu_{n}-2 \tau & 
  \leq \frac{1}{n} H_{\lambda_{n}}\left(\left. \bigvee_{h=0}^{n-1} f^{-h} \beta \right| \eta\right)  \\
  &\leq \frac{1}{q} H_{\mu_{n}}\left(\left.\bigvee_{i=0}^{q-1} f^{-i} \beta\right|f \beta^{\F}\right)+ \frac{3q}{n}\log \left(\operatorname{Card} \alpha\right). 
  \end{align*}
Letting $ n \to\infty  $ and using Proposition \ref{F-pressure} and proposition \ref{muntomu}, we derive 
  \begin{align*}
  P^{\mathcal{F}}_{\text{top}}(f, \phi, \overline{W^{\F}(x,\delta)})&=\lim _{n \rightarrow \infty} \frac{1}{n}\log Z_{n}^{\phi} \leq \lim _{n \rightarrow \infty}\left(\frac{1}{q}H_{\mu_{n}}\left(\left.\bigvee_{i=0}^{q-1} f^{-i}\beta\right| f\beta^{\F}\right)+\int_{M} \phi \,\rd \mu_{n}\right)+2 \tau \\
  & =\frac{1}{q}H_{\mu}\left(\left.\bigvee_{i=0}^{q-1} f^{-i} \beta \right|f\beta^{\F}\right)+\int_{M} \phi \,\rd \mu+2 \tau. 
  \end{align*}
  By taking $q\to\infty$, we have 
  \begin{align*}
    P^{\mathcal{F}}_{\text{top}}(f, \phi, \overline{W^{\F}(x,\delta)}) \leq h_{\mu}(f, \beta|f\beta^{\F})+\int \phi \,\rd \mu+2 \tau\leq h_{\mu}^{\F}(f)+\int_{M} \phi \,\rd \mu+2 \tau.
  \end{align*}
  Then the proof of this theorem  is completed by letting $\tau \to 0$. 
  \end{proof}

\section{Generating $\mathcal{F}$-Equilibrium States from Each $x$}\label{Generating F-Equilibrium States from Each x}

In this section, we prove that under certain conditions, for any $x\in M$ and $\delta>0$, 
$$
P^{\mathcal{F}}_{\text{top}}(f, \phi) 
=P^{\mathcal{F}}_{\text{top}}(f, \phi, \overline{W^{\F}(x,\delta)}) .
$$ 
For notational clarity, in this section, 
we will use $P^{u}_{\text{top}}(f, \phi)$ and $W^u(x,\delta)$ 
in place of $P^{\mathcal{F}}_{\text{top}}(f, \phi)$ and $W^{\F}(x,\delta)$, respectively. 
All other symbols will be treated similarly.


Using methods analogous to those in \cite[Proposition 3.1]{PP22-1}, 
we establish the following result. 

\begin{theorem}\label{F-pressure-2}
  Let $f: M \to M$ be a diffeomorphism on a closed Riemannian manifold $M$ 
satisfying conditions {\bf (C1)}-{\bf (C3)}, 
and $\phi:M \to \mathbb{R}$ a continuous function. 
Then for any $x\in M$ and $\delta>0$, 
\begin{align*}
  P_{\operatorname{top}}(f, \phi)
  &=\lim_{n\to\infty}\frac{1}{n}\log \left(\int_{f^n W^{u}_\delta(x)} e^{S_n\phi(f^{-n} y)}\,\rd \lambda_{f^n W^{u}_\delta(x)}\right)\\
  &= \lim_{n \rightarrow+\infty} \frac{1}{n} \log \left(\int_{W^{u}_\delta(x)}  e^{S_n \left(\phi-\Phi^{u}\right)\left(y\right)} \,\rd \lambda_{W^{u}_\delta(x)}(y)\right), 
\end{align*}
where $\Phi^u(x):=-\log \left| \det\left( Df|E^u_x \right) \right|$. 
Moreover, for any $x\in M$ and $\delta>0$, we have 
\[
P_{\operatorname{top}}(f,\phi)=P^{u}_{\operatorname{top}}(f,\phi)=P^{u}(f,\phi, \overline{W^u(x,\delta)}). 
\]  
\end{theorem}


Before stating our proof, we introduce the following auxiliary result.

\begin{lemma}\label{volume lemma}
Let $M$  be  a closed Riemannian manifold, and 
$\F$ a continuous foliation with $C^1$ leaves of dimension $k$.  
For any $\epsilon>0$, there exists $c,C>0$ such that 
\[
c\epsilon^k \leq \operatorname{vol}_{\F}\left( W^{\F}(x,\epsilon) \right) \leq C\epsilon^{k},\quad \text{ for any } \,x\in M. 
\]
\end{lemma}

\begin{proof}
We can find a finite open cover $\{U_i\}$ of $M$ with foliated charts $\phi_i: U_i \to \mathbb{R}^k \times \mathbb{R}^{n-k}$ where the leaves are mapped to $\mathbb{R}^k \times \{y\}$. 

On each chart $U_i$, the Riemannian metric $g$ has bounded coefficients: there exist $\lambda_i, \Lambda_i > 0$ such that for any leaf $L$ and $x \in U_i \cap L$,
\begin{align}\label{volume-(1)}
\lambda_i \|v\|_{\text{Eucl}} \leq \|v\|_g \leq \Lambda_i \|v\|_{\text{Eucl}} \quad \text{for } v \in T_xL,
\end{align}
where $\|\cdot\|_{\text{Eucl}}$ is the Euclidean norm. 
We also denote the Euclidean volume by $\text{vol}_{\text{Eucl}}$. 

In each chart $U_i$, for any $x \in U_i$, we have 
\[
B_{\text{Eucl}}(\phi_i(x), \lambda_i \epsilon) \subset \phi_i(W^{\mathcal{F}}(x,\epsilon)) \subset B_{\text{Eucl}}(\phi_i(x), \Lambda_i \epsilon),
\]
where $B^k_{\text{Eucl}}(x_0,r)$ is the Euclidean ball in $\mathbb{R}^k \times \{y\}$ centered at $x_0$ with radius $r$.
Combining this with \eqref{volume-(1)}, for any $x \in U_i$, we have 
\begin{align*}
\lambda_i^{2k}\epsilon^k \leq \lambda_i^k\text{vol}_{\text{Eucl}}(\phi_i(W^{\mathcal{F}}(x,\epsilon)))\leq \text{vol}_{\F}\left(W^{\F}(x,\epsilon) \right),\\
\text{vol}_{\F}\left(W^{\F}(x,\epsilon) \right)\leq \Lambda_i^k\text{vol}_{\text{Eucl}}(\phi_i(W^{\mathcal{F}}(x,\epsilon)))\leq \Lambda_i^{2k}\epsilon^k.  
\end{align*}
Taking $c = \min_i \lambda_i^{2k} > 0$ and $C=\max \Lambda_i^{2k}$ over the finite cover, 
we finish the proof. 
\end{proof}

In this paper, we only need the estimate $
c \leq \operatorname{vol}_{\F}\left( W^{\F}(x,\epsilon) \right) \leq C$.
This estimate plays an important role in our argument 
and appears in both the proof of Theorem C in \cite{HHW17} 
and the proof of Lemma 3.1 in \cite{HSX08}.

\begin{proof}[Proof of Theorem \ref{F-pressure-2}]
By Proposition \ref{F-pressure} and Corollary \ref{top pressure corollary}, 
we only need to prove the following inequality holds for any $x\in M$ and $\delta>0$: 
\begin{align*}
P_{\text{top}}(f, \phi)\leq \varliminf _{n \rightarrow+\infty} \frac{1}{n} \log \left(\int_{W^{u}(x, \delta)}  e^{S_n \left(\phi-\Phi^{\F}\right)\left(y\right)} \,\rd \lambda_{W^{u}_\delta(x)}(y)\right). 
\end{align*}

  Given $\tau>0$, 
  there exists $\epsilon_0=\epsilon_0(\tau)>0$ with $\epsilon_0<\delta$ such that the conclusion of 
   Lemma \ref{potential lemma} is satisfied. 

Fix $x\in M$ and $\delta>0,0<\epsilon<\epsilon_0$, by condition {\bf (C3)}, 
there exists an increasing function $h_{x,\delta}^\epsilon:\mathbb{N}\to \mathbb{N}$ with 
$\displaystyle \lim_{n\to\infty}h_{x,\delta}^\epsilon(n)n^{-1}=0$ 
such that for any $ n\in \mathbb{N}$, 
\begin{align*}
 \bigcup\limits_{z\in f^{h_{x,\delta}^\epsilon(n)} W^{u}(x,\delta)} W^{cs}(z, \epsilon g^{\epsilon}(n)^{-1})=M. 
\end{align*}
We choose a maximal set 
$S=\left\{x_{1}, \cdots, x_{N}\right\}\subset f^{n+h_{x,\delta}^\epsilon(n)} W^{u}(x,\delta)$ such that
\[
  W^u\left(x_{i}, \frac{\epsilon}{4}\right) \cap W^u\left(x_{j}, \frac{\epsilon}{4}\right)=\varnothing \quad \text{for any} \,\, i \neq j, \, 1\leq i,j\leq N
\]
and
  \[
    \bigcup_{i=1}^{N} W^{u}\left(x_{i}, \frac{1}{4}\epsilon\right)\subset f^{n} \left(W^{u}(x,\delta)\right), \quad f^{n} \left(\overline{W^u(x,\delta) }\right)\subset \bigcup_{i=1}^{N} W^{u}\left(x_{i}, \epsilon\right).
  \]
For any $z\in M$, there exists $y\in f^{h_{x,\delta}^\epsilon(n)} W^{u}(x,\delta)$ satisfying 
$z\in W^{cs}(y, \epsilon g^\epsilon(n)^{-1})$. 
For each $0\leq j<n$, by condition {\bf (C3)}, we have 
\begin{align*}
  d^{cs}(y,z)< \epsilon g^\epsilon(n)^{-1}\leq \epsilon g^\epsilon(j)^{-1} \Rightarrow d^{cs}(f^jz, f^jy)<\epsilon. 
\end{align*} 
By the maximality of the set $S$ and property (3) of Definition \ref{def: expanding foliation}, 
we conclude that there exists $x_{i}\in S$ such that 
\begin{align*}
\max\limits_{0\leq j<n}d^{u}\left(f^jy, f^j\left(f^{-n} x_i\right)\right)=d^{u}\left(f^ny, x_i\right)<\epsilon.
\end{align*} 
Therefore, for the chosen point $z\in M$, we derive 
\begin{align*}
\max\limits_{0\leq j<n}d\left( f^jz, f^j\left(f^{-n}x_{i}\right)\right) &\leq 
\max\limits_{0\leq j<n}d\left(f^jz, f^jy\right)+\max\limits_{0\leq j<n}d\left( f^jy, f^j\left(f^{-n}x_{i}\right)\right)\\
&\leq \max\limits_{0\leq j<n}d^{cs}\left(f^jz, f^jy\right)+\max\limits_{0\leq j<n}d^{u}\left( f^jy, f^j\left(f^{-n}x_{i}\right)\right) < 2\epsilon. 
\end{align*}
Consequently, by setting $y_i:=f^{-n} x_i\in f^{h_{x,\delta}^\epsilon(n)}W^u(x,\delta)$ for $1\leq i\leq N$, 
the points $\left\{y_{1}, \ldots, y_{N}\right\}$ form an $(n,2\epsilon)$-spanning set for $M$. 
Let $\text{vol}_u$ denote the volume induced by the Riemannian structure on the foliation $W^u$.  
By Lemma \ref{volume lemma}, there exists $c_1>0$ such that 
 $\inf\limits_{z\in M} \text{vol}_u\left(W^u\left(z,\left(\frac{\epsilon}{4}\right)\right)\right)\geq c_1$. 
Therefore, we have 
\begin{align*}
 S_{d}(\phi,n,2\epsilon)  &\leq \sum_{i=1}^{N} e^{S_n \phi\left(y_{i}\right)}=\sum_{i=1}^{N} e^{S_n \phi\left(f^{-n}x_{i}\right)}\nonumber \\
& \leq \sum_{i=1}^{N} \frac{e^{n\tau}}{\text{vol}_u\left(W^u\left(f^n y_{i}, \frac{\epsilon}{4}\right)\right)} \int\limits_{W^u\left(f^n y_{i}, \frac{\epsilon}{4}\right)} e^{S_n \phi\left(f^{-n}z\right)} \,\rd \lambda_{f^{n+h_{x,\delta}^\epsilon(n)} W^{u}(x,\delta)}(z)\nonumber\\
& \leq \frac{1}{c_1} e^{n \tau} \int\limits_{f^{n+h_{x,\delta}^\epsilon(n)} W^{u}(x,\delta)} e^{S_n\phi\left(f^{-n} z\right)} \,\rd \lambda_{f^{n+h_{x,\delta}^\epsilon(n)}W^{u}(x,\delta)}(z)\nonumber\\
& \leq \frac{1}{c_1} e^{n \tau+h_{x,\delta}^\epsilon(n)\|\phi\|} \int\limits_{f^{n+h_{x,\delta}^\epsilon(n)} W^{u}(x,\delta)} e^{S_{n+h_{x,\delta}^\epsilon(n)}\phi\left(f^{-(n+h(n))} z\right)} \,\rd \lambda_{f^{n+h_{x,\delta}^\epsilon(n)}W^{u}(x,\delta)}(z),
\end{align*}
where $\|\phi\|:=\sup\limits_{x\in M}|\phi(x)|<\infty$. 
Using $\displaystyle \lim_{n\to\infty}h_{x,\delta}^\epsilon(n)n^{-1}=0 $, we conclude
\begin{align*}
 P_{\text{top}}(f,\phi) &=\lim _{\epsilon \rightarrow 0} \varliminf_{n \rightarrow+\infty} \frac{1}{n} \log S_{d}(\phi, n, 2 \epsilon)\\
  &\leq \varliminf_{n \rightarrow+\infty} \frac{1}{n} \log \int\limits_{f^n W^{u}(x,\delta)} e^{S_n\phi\left(f^{-n} z\right)} \,\rd \lambda_{f^{n}W^{u}(x,\delta)}(z)+\tau.
\end{align*} 
Since $\tau$ is arbitrary, the theorem follows.
\end{proof}

  \section{Applications}

In this section, 
we present several classes of diffeomorphisms that satisfy the hypotheses of Theorem \ref{F-pressure-2}. 
As a consequence, for these dynamical systems, 
equilibrium states for continuous potentials 
can be explicitly constructed through the local leaves of the expanding foliation.

\subsection{Proof of Theorem \ref{exponential mixing lemma-theorem}}\label{Exponential Mixing Diffeomorphisms}

We begin by recalling some basic facts about exponential mixing. 

\begin{definition}
Let $f : M \to M$ be a $C^{1+\alpha}$ diffeomorphism on a closed Riemannian
 manifold $M$ and $\mu$ an $f$-invariant Borel probability measure. 
The diffeomorphism $f$ is {\bf exponential mixing} with respect to the measure $\mu$ 
  if there are constants $C>0, \mathbf{r}>0$ and $\eta_{\mathbf{r}}>0$ such that for all $\phi, \psi \in C^{\mathbf{r}}(M)$, 
\begin{align}\label{exponential mixing}
\left|\int_M \phi(x) \psi\left(f^n x\right) \,\rd \mu-\int_M \phi \,\rd \mu \int_M \psi\, \rd \mu\right| \leq C e^{-\eta_{\mathbf{r}} n}\|\phi\|_{\mathbf{r}}\|\psi\|_{\mathbf{r}}, 
\end{align}
where $\|\cdot\|_{\mathbf{r}}$ is the norm on $C^{\mathbf{r}}(M)$. 
\end{definition}

\begin{lemma}[\cite{DKR} Lemma B.1]\label{Lemma 2.1.}
 Suppose that $\mu$ is a smooth measure. If \eqref{exponential mixing} holds for some $\mathbf{r}>0$ then it holds for all $\tilde{\mathbf{r}}>0$ (with possibly a different exponent $\eta_{\tilde{\mathbf{r}}}$ ).
\end{lemma}

Throughout this paper, we always use equation \eqref{exponential mixing} for the case $\mathbf{r}=1$.

The following proposition establishes that the combination of exponential mixing and condition {\bf (C2)} yields condition {\bf (C3)}. 
The proof adapts the approximation scheme developed in \cite[Lemma B.2]{DKR}.

\begin{proposition}\label{exponential mixing proposition}
  Let $ M $ be a closed Riemannian manifold and 
  $ f: M \rightarrow M $ a diffeomorphism preserving a smooth measure $\mu$. 
  Suppose that 
  \begin{itemize}
\item $f$ satisfies conditions {\bf (C1)} and {\bf (C2)} ;
\item the continuous foliation $W^{cs}$ has $C^1$ leaves;
\item $f$ is exponential mixing with respect to the smooth measure $\mu$. 
  \end{itemize}
  Then condition {\bf (C3)} holds. 
      \end{proposition}

    \begin{proof}
    We face a challenge not addressed in \cite[Lemma B.2]{DKR}. 
    While that proof demonstrates a nonempty intersection 
    by demonstrating positive intersection volume, 
    in our case, even if $f^{h\left(n\right)}\left(B^u\left(x,\delta\right)\right)$
   intersects $B^{cs}\left(y,g\left(n\right)^{-1}\right)$, their intersection volume may still be $0$. 
   To overcome this difficulty, 
 we slightly expand these sets to ensure positive intersection volume. 
These expansions must be chosen sufficiently small, together with slight truncations of the original sets, in order to guarantee that intersections of the expanded sets imply intersections of the originals.

Fix $x\in M$ and $\delta, \epsilon_0>0$. 
We use $g,h$ as shorthand for functions $g^{\epsilon_0}$ and $h_{x,\delta}^{\epsilon_0}$ 
in conditions {\bf (C2)} and {\bf (C3)}, respectively. 
To prove this lemma, we only need to show that for all $y\in M$,  
    \begin{align}\label{exponential mixing proposition-(1)}
    f^{h\left(n\right)}\left(\widehat{W}^{u}\left(x,\delta\right)\cap \widehat{W}^{cs}\left(y,\epsilon_0 g\left(n\right)^{-1}\right)\right)\neq \varnothing,
    \end{align}
    where
      \begin{align*}
           \widehat{W}^u\left(x,\delta\right)&:=\{y\in M: d\left(y,W^u\left(x,\frac{1}{3}\delta\right)\right)<\frac{\epsilon_0}{3}g\left(n\right)^{-1}\} ;\\
            \widehat{W}^{cs}\left(y,g\left(n\right)^{-1}\right)&:=\{y\in M: d\left(y,W^{cs}\left(y,\frac{\epsilon_0}{3}g\left(n\right)^{-1}\right)\right)<\frac{\epsilon_0}{3}g\left(n\right)^{-1}\}.
      \end{align*}
      We denote by $m_u$ and $m_{cs}$ the dimensions of the 
       foliations $W^{u}$ and $W^{cs}$, respectively. 
       For $\epsilon>0$ and $A\subset M$, we denote by
       $V_\epsilon\left(A\right):=\{x\in M: d\left(x,A\right)<\epsilon\}$ the $\epsilon$-neighborhood of $A$, 
       and  $1_A$ the characteristic function of $A$.   
      
       We fix parameters $r=r(n) $ and $r'=r'(n)$ satisfying $ r(n),r'(n)<g\left(n\right)$; 
       their exact values will be determined later. 
      Let $\phi_{r}$ be a Lipschitz function that is $1$ on $\widehat{W}^u\left(x,\delta\right)$, $0$ outside $V_r\left(\widehat{W}^u\left(x,\delta\right)\right)$, 
      and has Lipschitz norm of order $O\left(1 / r\right)$. 
          Similarly, let $\psi_{r'}$ be a Lipschitz function that is $1$ 
          on $\widehat{W}^{cs}\left(y, g\left(n\right)^{-1}\right)$, 
          $0$ outside $V_{r'}\left(\widehat{W}^{cs}\left(y,g\left(n\right)^{-1}\right)\right)$, 
          and has Lipschitz norm of order $O\left(1 / r'\right)$. 
          By the definitions of $\phi_r$ and $\Psi_{r'}$, we obtain 
          \begin{align*}
            & \|\phi_r\|_{L^2}^2=O\left(g\left(n\right)^{-m_{cs}}\right), \quad \|\psi_{r'}\|=O\left(g\left(n\right)^{-\left(m_u+m_{cs}\right)}\right); \\
          &\left\|\phi_{r}\right\|_{L i p}=O\left(\frac{1}{r}\right),\quad \left\|\psi_{r}\right\|_{L i p}=O\left(\frac{1}{r'}\right);\\
          &\left\|\phi_{r}-1_{\widehat{W}^u\left(x, \delta\right)}\right\|_{L^2}^2=O\left(rg\left(n\right)^{-m_{cs}}+rg\left(n\right)^{-\left(m_{cs}-1\right)}\right)=O\left(rg\left(n\right)^{-\left(m_{cs}-1\right)}\right); \\
          & \left\|\psi_{r'}-1_{\widehat{W}^{cs}\left(y, g\left(n\right)^{-1}\right)}\right\|_{L^2}^2=O\left(g\left(n\right)^{-\left(m_u+m_{cs}-1\right)}r'\right).
          \end{align*}
  Combining the exponential mixing property with above equalities, we obtain
          \begin{align*}
          &\mu\left(f^{h\left(n\right)}\left(\widehat{W}^u\left(x, \delta\right) \right)\cap \widehat{W}^{cs}\left(y, g\left(n\right)^{-1}\right)\right)\\
          =&\mu\left(\phi_{r}\left(\psi_{r'}\circ f^{h\left(n\right)}\right)\right)+ O\left(\|\phi_r\|_{L^2}\left\|\psi_{r'}-1_{\widehat{W}^u\left(x,\delta\right)}\right\|_{L^2}+\left\|\phi_r-1_{\widehat{W}^{cs}\left(y,g\left(n\right)^{-1}\right)}\right\|_{L^2}\|\psi_{r'}\|_{L^2} \right)\\
          =&\mu\left(\phi_{r}\right) \mu\left(\psi^{cs}_{r'}\right)+O\left(\sqrt{r'g\left(n\right)^{-\left(m_u+2m_{cs}-1\right)} }+\sqrt{r g\left(n\right)^{-\left(m_u+2m_{cs}-1\right)} }+r^{-1}r'^{-1}e^{-\eta_1 h\left(n\right)}\right)  \\
         =& \mu\left(\widehat{W}^u\left(x, \delta\right) \right)\mu\left(\widehat{W}^{cs}\left(y, g\left(n\right)^{-1}\right)\right)\\
         &\quad\quad +O\left(\sqrt{r'g\left(n\right)^{-\left(m_u+2m_{cs}-1\right)} }+\sqrt{r g\left(n\right)^{-\left(m_u+2m_{cs}-1\right)} }+r^{-1}r'^{-1}e^{-\eta_1 h\left(n\right)}\right) . 
          \end{align*}
          By choosing $r=r'=e^{-\frac{2}{5}\eta_1 h\left(n\right)}g\left(n\right)^{\frac{1}{5}\left(m^{u}+2m_{cs}-1\right)}$,
         we ensure that any terms in $ O\left(\cdot\right)$ are of the same order, thus we obtain 
          \begin{align*}
            &\mu\left(f^{h\left(n\right)}\left(\widehat{W}^u\left(x, \delta\right) \right)\cap \widehat{W}^{cs}\left(y, g\left(n\right)^{-1}\right)\right)\\
            &=\mu\left(\widehat{W}^u\left(x, \delta\right) \right)\mu\left(\widehat{W}^{cs}\left(y, g\left(n\right)^{-1}\right)\right)+O\left( e^{-\frac{1}{5}\eta_1 h\left(n\right)}g\left(n\right)^{-\frac{2}{5}\left(d^{u}+2m_{cs}-1\right)} \right).
          \end{align*}
        The first term in the right-hand side is of order $O\left(g\left(n\right)^{-\left(m_u+2m_{cs}\right)}\right)$, 
        and it is larger than the last term of the right hand when 
        \[
         h\left(n\right)\geq 3\eta_1^{-1} \left(m_u+2m_{cs}-1\right)\log g\left(n\right).
         \] 
        By defining
        $$
        h\left(n\right):=\max\left\{\lfloor \eta^{-1}\left( m^u+2m^{cs}+4 \right)\log g\left(n\right) \rfloor, \lfloor 10\left(m_u+2m_{cs}-1\right)\eta_1  \log g\left(n\right) \rfloor\right\},
        $$ 
        we conclude that $\lim\limits_{n\to\infty} h\left(n\right)n^{-1}=0$ and $r,r'<g\left(n\right)$ for all sufficiently large $n$. 
      
        Since the product of volumes satisfies 
        $$
        \mu\left(\widehat{W}^u\left(x, \delta\right) \right)\mu\left(\widehat{W}^{cs}\left(y, g\left(n\right)^{-1}\right)\right)>0, 
        $$
       it follows that  
       $$
       \mu\left(f^{h\left(n\right)}\left(\widehat{W}^u\left(x, \delta\right) \right)\cap \widehat{W}^{cs}\left(y, g\left(n\right)^{-1}\right)\right)>0.
       $$
        This establishes \eqref{exponential mixing proposition-(1)}, 
        thereby completing the proof. 
        \end{proof}

We are now ready to prove Theorem \ref{exponential mixing lemma-theorem}. 

\begin{proof}[Proof of Theorem \ref{exponential mixing lemma-theorem}]
Theorem \ref{exponential mixing lemma-theorem} directly follows from proposition \ref{exponential mixing proposition} and Theorem \ref{F-pressure-2-corollary}. 
\end{proof}

Exponential mixing plays a crucial role in the
 study of statistical properties of dynamical systems.
 In \cite{DKR}, Dolgopyat, Kanigowski and Rodriguez Hertz prove that for systems preserving a smooth measure,
 exponential mixing implies Bernoulli. 
Recently, Maldonado extends the work in \cite{DKR}, showing that systems together with the SRB measure,
 exponential mixing  also implies Bernoulli \cite{Mal25}.

\subsection{Proof of Theorem \ref{Katok-existence F-es}}\label{Equilibrium states for Katok map}

In this subsection, we prove Theorem \ref{Katok-existence F-es}. 
We begin by recalling some results on the uniqueness of the equilibrium state for the Katok map.

  \begin{theorem}[\cite{PSZ19} Theorem 3.2, \cite{Wa21} Theorem 1.1]\label{Katok-equilibrium}
  Given the Katok map $G_{\mathbb{T}^2}$, if the H\"older continuous function $\varphi: \mathbb{T}^2\to \mathbb{R}$ and 
  $\varphi(0)<P(\varphi)$, where $0$ is the origin, then there is a unique equilibrium state for $\varphi$. 
  
  For the geometric $t$-potential $\varphi_t=-t\log|DG_{\mathbb{T}^2}|E^u(x)|$, we have 
  \begin{itemize}
          \item for \( t < 1 \), there exists a unique equilibrium measure $\mu_t$;
         
      \item for \( t = 1 \), there exist two equilibrium measures associated to $\varphi_1$, 
      namely, the Dirac measure at the origin $\delta_0$ and the area $m$.
      
      \item for \( t > 1 \), $\delta_0$ is the unique equilibrium measure associated to $\varphi_t$.
  \end{itemize}
  \end{theorem}

  The following results are immediate corollaries of the previous theorems. 

  \begin{corollary}\label{Katok-tneq1}
    For the H\"older continuous function $\varphi: \mathbb{T}^2\to \mathbb{R}$ with $\varphi(\underline{0})<P(\varphi)$ 
and the geometric $t$-potential $\varphi_t=-t\log|DG_{\mathbb{T}^2}|E^u(x)|$ with $t\in (-\infty, 1)\cup(1,\infty)$, 
      for any $x\in \mathbb{T}^2$ and $\delta>0$, 
      the constructed measures $\mu_n$ in Theorem \ref{Katok-existence F-es}
      converge to the unique equilibrium states in the weak* topology. 
    \end{corollary}


\begin{corollary}\label{Katok-t1}
  For the geometric potential $\varphi=-\log|DG_{\mathbb{T}^2}|E^u(x)|$, 
    for any $x\in \mathbb{T}^2$ and $\delta>0$, 
    the probability measures 
    \begin{align*}
    \mu_{n}:=\frac{1}{n} \sum_{k=0}^{n-1} \left( G_{\mathbb{T}^2}^k \right)_{*} \left( \frac{\lambda_{W^{u}(x,\delta)}}{\int_{W^{u}(x,\delta)} 1 \,\rd \lambda_{W^{u}(x,\delta)(z)}}\right)
    \end{align*}
    converge to the equilibrium states in the weak* topology. 
    Each equilibrium state is a convex combination of $m$ and $\delta_0$. 
  \end{corollary}

 Katok map is a $C^{\infty}$ nonuniformly hyperbolic diffeomorphism 
of the $2$-torus $\mathbb{T}^2$. 
Consider the matrix $A=\begin{pmatrix} 2&1\\1&1\end{pmatrix}$, 
which induces a linear transformation $T$ of the two-dimensional torus 
$\mathbb{T}^2=\mathbb{R}^2/\mathbb{Z}^2$ with the Lebesgue measure $m$. 
The Katok map $G_{\mathbb{T}^2}$ is a slowdown of the linear Anosov map $T$ near the origin, 
and it is a local perturbation. 
We do not provide the explicit definition of the Katok map and 
 refer the reader to \cite{BP23, Ka79} for precise definitions. 
Below, we list the relevant properties of the Katok map.

\begin{proposition}[\cite{BP23,Ka79}]\label{Katok property}
The Katok map has the following properties: 
\begin{itemize}
\item[(1)] It is topologically conjugated to $ T $ via a homeomorphism $ \phi $, $T \circ \phi=\phi \circ G_{\mathbb{T}^2}$. 

\item[(2)] It admits two transverse invariant continuous stable and 
unstable distributions $E^u(x)$ and $E^s(x)$ that integrate to continuous, 
uniformly transverse and invariant foliations $W^u(x)$ and $W^s(x)$ with smooth leaves. 
Moreover, they are the images under the conjugacy map $\phi$ of the stable and unstable foliations for $T$, respectively.

\item[(3)] Almost every $x$ with respect to $m$ has two non-zero Lyapunov exponents, one
 positive in the direction of $E^u(x)$ and the other negative in the direction of $E^s(x)$.
Moreover, the only ergodic invariant measure with zero Lyapunov exponents is the Dirac measure at the origin \( \delta_0 \).

  \item[(4)] It is ergodic with respect to the Lebesgue measure \( m \).
 
\end{itemize}
\end{proposition}

For the Katok map $G_{\mathbb{T}^2}$, 
we denote by $d^u_{G_{\mathbb{T}^2}}$ and $d^s_{G_{\mathbb{T}^2}}$ 
the induced metrics on $W^u, \,W^s$ respectively.
Similarly, for the Anosov diffeomorphism $T$, 
we write $d^u_T, d^s_T$ for the corresponding induced metrics on its unstable and stable foliations.  

Given $\epsilon>0$ with $\epsilon<\delta_0$, we define the constant $\tilde{\epsilon}$ by 
\begin{align*}
\tilde{\epsilon}:=\inf\left\{ d(y,z)\,|\, y,z\in W^s(x),\, x\in\mathbb{T}^2,\, \frac{\epsilon}{2} \leq d^s_{G_{\mathbb{T}^2}}(y,z)\leq \epsilon\right\}. 
\end{align*} 
We claim that $\tilde{\epsilon}>0$ if $\epsilon>0$. 
We proceed by contradiction and assume that $\tilde{\epsilon}=0$ for some $\epsilon>0$. 
By the continuity of foliation $W^s$, there exists a point $y\in \mathbb{T}^2$ 
and a sequence $\{z_i\}_{i\in\mathbb{N}}$ satisfying $z_i\in W^s(y)$ and  $\displaystyle \lim_{i\to\infty}d(y,z_i)\to 0$. 
This contradicts the fact that the set $A=\{z\in W^s(y)\,|\, \frac{\epsilon}{2}\leq d^s_{G_{\mathbb{T}^2}}(y,z)\leq \epsilon\}$ 
is compact and $d(y,z)>0$ for all $z\in A$. 
Therefore, we conclude that $\epsilon>0$ implies $\tilde{\epsilon}>0$. 

Since $\phi:M\to \mathbb{R}$ is a homeomorphism on a compact space $M$, 
there exists functions $ \theta:\mathbb{R}_+\to\mathbb{R}_+ $  
and $ \eta:\mathbb{R}_+\to\mathbb{R}_+ $
such that for any $ x,y\in \mathbb{T}^2$ and $ \epsilon>0$, 
\begin{align*}
d(\phi(x),\phi(y))<\theta(\epsilon)\Rightarrow d(x,y)<\epsilon, \quad d(x,y)<\epsilon\Rightarrow d(\phi(x),\phi(y))<\eta(\epsilon). 
\end{align*}

By Proposition \ref{Katok property} (1) and (2), the local product structure is inherited by $G_{\mathbb{T}^2}$, i.e. 
there exist $\delta_0>0$ and a function $r(\delta)$ such that for any $x,y\in M$,
    \begin{itemize}
\item $W^u(x, \delta)\cap W^s(y, \delta)$ contains at most one point when $0<\delta<\delta_0$; 
\item if $d(x,y)<r(\delta)$, then $W^u(x, \delta)\cap W^s(y, \delta)\neq \varnothing$ for $0<\delta<\delta_0$. 
\end{itemize}

    We are now ready to prove Theorem \ref{Katok-existence F-es}. 

\begin{proof}[Proof of Theorem \ref{Katok-existence F-es}]
By Proposition \ref{Katok property}(2),for any $x\in \mathbb{T}^2$ and $y\in W^u(x)$, 
$\phi(x)$ and $\phi(y)$ lie on the same unstable leaf of the linear Anosov diffeomorphism $T$. 
We obtain 
\begin{align*}
  d(T^n(\phi(x)), T^n(\phi(y)))&=d^u_{T}(T^n(\phi(x)), T^n(\phi(y)))=\left(\frac{3+\sqrt{5}}{2}\right)^n d^u_{T}(\phi(x), \phi(y))\\
  &=\left(\frac{3+\sqrt{5}}{2}\right)^n d(\phi(x), \phi(y)). 
\end{align*}
Fix a small $a>0$. 
Since the homeomorphism $\phi$ maps an curve of length $a$ to a curve of length at least $\theta(a)$, 
we obtain 
\begin{align*}
  d^u_{G_{\mathbb{T}^2}}(G_{\mathbb{T}^2}^n(x), G_{\mathbb{T}^2}^n(y))&
  \geq   d(G_{\mathbb{T}^2}^n(x), G_{\mathbb{T}^2}^n(y)) 
  = d(\phi^{-1} \circ T^n\left(\phi(x)\right), \phi^{-1}\circ T^n \left(\phi(y)\right)) \\
  & \geq \frac{\theta(a)}{2a} d(T^n\left(\phi(x)\right),  T^n \left(\phi(y)\right))
  = \frac{\theta(a)}{2a}  \left(\frac{3+\sqrt{5}}{2}\right)^n d(\phi(x), \phi(y)). 
\end{align*}
Combining this inequality with Proposition \ref{Katok property} (2), 
we conclude that the Katok map satisfies condition {\bf (C1)}.

 Fix $\epsilon>0$. For any $x\in M$ and $y\in W^s(x,\epsilon)$, we have 
  \begin{align*} 
 d(x,y)\leq d^{s}_{G_{\mathbb{T}^2}} (x,y)< \epsilon \Rightarrow  d(\phi(x),\phi(y))\leq \eta(\epsilon). 
  \end{align*}
  There exists an integer $N\in\mathbb{N}$ such that for any $ n\geq N$, 
 \begin{align*} 
  d\left(\phi\left( G_{\mathbb{T}^2}^n(x) \right), \phi\left( G_{\mathbb{T}^2}^n(y) \right)\right)=  d\left(T^n\left( \phi(x) \right), T^n\left( \phi(y) \right)\right)\leq \eta(\epsilon) \left(\frac{3-\sqrt{5}}{2}\right)^n<\theta(\tilde{\epsilon}). 
 \end{align*}
This implies that for any $ n\geq N$, $  d\left( G_{\mathbb{T}^2}^n(x) ,  G_{\mathbb{T}^2}^n(y) \right)<\tilde{\epsilon} $ 
and $d^s_{G_{\mathbb{T}^2}}\left( G_{\mathbb{T}^2}^n(x) ,  G_{\mathbb{T}^2}^n(y) \right)<\epsilon $. 
Moreover, there exists a constant $N_\epsilon $ such that for any $0\leq j\leq N$
$$ 
d^s_{G_{\mathbb{T}^2}}(x,y)<\epsilon N_\epsilon \Rightarrow d^s_{G_{\mathbb{T}^2}}(G_{\mathbb{T}^2}^jx,G_{\mathbb{T}^2}^jy)<\epsilon. 
$$
Taking $g^\epsilon(x)\equiv N_\epsilon$, we conclude that the condition {\bf (C2)} is satisfied.

Fix $x\in M$ and $\delta, \epsilon>0$. 
To verify condition {\bf (C3)}, without loss of generality, 
we assume that 
$\epsilon<\min\left\{\frac{\delta_0}{2},\frac{\delta}{2}\right\}$.
By Proposition \ref{Katok property} (2), 
there exists $m\in\mathbb{N}$ such that for all $n\geq m$, 
\begin{itemize}
\item the set $G_{\mathbb{T}^2}^n(W^u(x,\frac{\delta}{2}))$ is 
$ r\left(\frac{\epsilon}{2}N_\epsilon^{-1}\right)$-dense;
\item for any $z\in G_{\mathbb{T}^2}^n\left(W^u(x,\frac{\delta}{2})\right)$, 
we have $W^u(z,\frac{\epsilon}{2}N_\epsilon^{-1})\subset G_{\mathbb{T}^2}^n\left(W^u(x,\delta)\right)$. 
\end{itemize}
Thus, for any $y\in\mathbb{T}^2$, 
there exists $z \in G_{\mathbb{T}^2}^n\left(W^u(x,\frac{\delta}{2})\right)$ 
satisfying $d(z,y)< r\left(\frac{\epsilon}{2}N_\epsilon^{-1}\right)$. 
Applying the local product structure, we obtain $z_0\in\mathbb{T}^2$ such that 
\begin{align*}
z_0\in W^u\left(z, \frac{\epsilon}{2}N_\epsilon^{-1}\right)\cap W^s\left( y,\frac{\epsilon}{2} N_\epsilon^{-1}\right) \subset G_{\mathbb{T}^2}^n\left( W^u(x,\delta) \right)\cap W^s(y,\epsilon g^\epsilon(n)). 
\end{align*}
This shows that the Katok map satisfies the conditions {\bf (C3)}. 
According to Theorem \ref{F-pressure-2-corollary}, 
we finish the proof of this theorem. 
\end{proof}

\subsection{Proof of Theorem \ref{Almost Anosov-existence F-es}}\label{Almost Anosov}

We begin by recalling the definition of the ``almost Anosov'' $(f,M)$ from \cite{HY95}. 
Let \( M \) be a two-dimensional closed Riemannian manifold, 
and let \( f :M\to M \) be a topologically transitive $C^2$ diffeomorphism satisfying 
the following conditions: 
\begin{itemize}
\item  \( f \) has a fixed point \( p \), i.e., \( f(p) = p \);
\item  There exists a constant \( \kappa^s < 1 \), a continuous function \( \kappa^u \) with
\[
\kappa^u(x) 
\begin{cases} 
= 1, & \text{at } x = p; \\ 
> 1, & \text{elsewhere}, 
\end{cases}
\]
and a decomposition of the tangent space $T_xM = E_x^u \oplus E_x^s$ such that
\[
|Df_x v| \leq \kappa^s |v|, \quad \forall v \in E_x^s,
\]
\[
|Df_x v| \geq \kappa^u(x) |v|, \quad \forall v \in E_x^u,
\]
and
\[
|Df_p v| = |v|, \quad \forall v \in E_p^u.
\]
\end{itemize}

  We introduce some definitions and propositions from \cite{HY95}. 

  \begin{proposition}[\cite{HY95} Lemma 2.1, Proposition 2.2]\label{HY95-Proposition 2.2}
    The maps $x\to E^u_x$ and $x\to E^s_x$ are continuous.  
    Moreover, there exist two continuous foliations \( W^u \) and \( W^s \) on \( M \) tangential to \( E^u \) and \( E^s \) respectively satisfying: 
    \begin{enumerate}
        \item The leaf of \( W^s \) through \( x \), denoted by \( W^s(x) \), is the stable manifold at \( x \), i.e.
        \[
          W^u(x) = \{ y \in M : \exists \,C=C_y\, \,s.t. \, d(f^{n}x, f^{n}y) \leq C(\kappa^s)^n,\, \forall n\geq 0 \}.
          \]
        
        \item The leaf of \( W^u \) through \( x \), denoted by \( W^u(x) \), 
        is the weak unstable manifold at \( x \), i.e.
        \[
        W^u(x) = \{ y \in M : \lim_{n \to \infty} d(f^{-n}x, f^{-n}y) = 0 \}.
        \]
        \item There exist constants \( \beta > 0 \) and \( D > 0 \) such that for all \( x \in M \), 
        if \( W_\beta^u(x) \) is the component of \( W^u(x) \cap \exp_x E_x(\beta) \) containing \( x \), 
        then \( \exp_x^{-1} W_\beta^u(x) \) is the graph of a $C^2$ function \( \phi_x^u : E_x^u(\beta) \to E_x^s(\beta) \) with 
        \( \phi_x^u(0) = 0 \) and \( \|\phi_x^u\|_{C^2} \leq D \). 
        That is, the foliaiton $W^u$ has $C^2$ leaves. 
        The analogous statement holds for \( W_\beta^s(x) \).
    \end{enumerate}
    \end{proposition}

  After Proposition 2.5 in \cite{HY95}, it is shown that the system $(f, M)$ has a local product structure,  i.e., 
    there exist $\delta_0>0$ and a function $r(\delta)$ such that for any $x,y\in M$,
    \begin{itemize}
\item $W^u_\delta(x)\cap W^s_\delta(y)$ 
contains at most one point when $0<\delta<\delta_0$; 
\item if $d(x,y)<r(\delta)$, then $W^u_\delta(x)\cap W^s_\delta(y)\neq \varnothing$ for $0<\delta<\delta_0$. 
\end{itemize}


    \begin{proposition}[\cite{HY95} Proposition 2.6]\label{HY95-Proposition 2.6}
  $W^u(p)$ and $W^s(p)$ are both dense in $M$. 
         \end{proposition}


         First, we consider the properties of weak unstable leaves of points near the point $p$.

         \begin{lemma}\label{almost dense-1}
          Given $\alpha>0$,  
        there exists an integer $N=N(\alpha)$ such that 
        for any $x\in M$ with  $d(x,p)<\frac{1}{2}   r\left(\frac{\alpha}{2}\right)$, $y\in M$ and $n\geq N$, 
        \begin{align*}
          f^n(W^u(x,\alpha))\cap W^s(y,\alpha)\neq \varnothing. 
        \end{align*}
        \end{lemma}
          
        \begin{proof}
          We set $\delta=\min\{\delta_0, \frac{\alpha}{2}, \frac{1}{4}   r\left(\frac{\alpha}{2}\right)\}$. 
      By Proposition \ref{HY95-Proposition 2.6}, 
      there exists an integer $N=N(\alpha)$ such that $f^n\left(W^u(p,\delta)\right)$ is $r(\delta)$-dense for any $n\geq N$, i.e., 
      \begin{align*}
      M=\bigcup_{z\in f^{n}\left(W^u(p,\delta)\right)} B(z,r(\delta)). 
      \end{align*}
      For any $y\in M$, there exists $z\in f^{n}\left(W^u(p,\delta)\right)$ with $d(z,y)<r(\delta)$. 
      Using the local product structure, there exists $ w\in W^u(z,\delta) \cap W^s(y,\delta)$. 
     Then we have 
     \begin{align*}
     f^{-n}(w)\in  W^u(f^{-n}(z),\delta)\subset W^u(p,2\delta)\subset W^u\left(p,\frac{1}{2}   r\left(\frac{\alpha}{2}\right)\right). 
     \end{align*} 
     For any $x\in M$ with $d(x,p)<\frac{1}{2}   r\left(\frac{\alpha}{2}\right)$, 
     applying the triangle inequality, it follows that  
     \begin{align*}
      d(f^{-n}(w),x)\leq d(f^{-n}(w),p)+d(p,x)<\frac{1}{2}   r\left(\frac{\alpha}{2}\right)+\frac{1}{2}   r\left(\frac{\alpha}{2}\right)=   r\left(\frac{\alpha}{2}\right). 
     \end{align*}
     By local product structure, there exists $v\in W^{s}(f^{-n}(w),\frac{\alpha}{2})\cap W^u(x,\frac{\alpha}{2})$. 
     Thus,  
     \begin{align*}
     f^n(v)\in f^n\left(W^u\left(x,\frac{\alpha}{2}\right)\right)\cap W^s\left(w,\frac{\alpha}{2}\right) \subset f^n\left(W^u(x,\alpha)\right)\cap W^s(y,\alpha). 
     \end{align*}
     This finishes the proof.
        \end{proof}

We extend the previous lemma to arbitrary points in $M$. 

\begin{lemma}\label{almost dense-2}
  Given $x\in M$ and $\delta>0$, 
  there exists an integer $N_0=N_0(x,\delta)>0$ such that for any $y\in M$ and $n\geq N_0$, 
  \begin{align*}
  f^{n}\left(W^{u}(x,\delta)\right)\cap W^{s}(y,\delta)\neq \varnothing.
  \end{align*}  
\end{lemma}

\begin{proof}
  Fix $x \in M$ and $\delta>0$. 
  By Proposition \ref{HY95-Proposition 2.6}, 
   the stable leaf $W^s(p)$ is dense. Therefore, 
  there exists $N = N(\delta) \in \mathbb{N}$ such that for all $n \geq N$, 
  the set $f^{-n}\left(W^s(p,\frac{\delta}{2})\right)$ is $r(\frac{\delta}{2})$-dense. 
   For each $n$, there exists $z \in f^{-n}\left(W^s\left(p,\frac{\delta}{2}\right)\right)$ 
   satisfying $d(x, z) < r\left(\frac{\delta}{2}\right)$. 
   Applying local product structure, 
   there exists $x_0\in W^u\left(x,\frac{\delta}{2}\right)\cap W^s\left(z,\frac{\delta}{2}\right)$. 
  
Since $x_0\in W^s(z,\frac{\delta}{2})\subset W^s(p)$, 
there exists $N_1\in\mathbb{N}$ such that 
$d(f^q(x_0),p)<\frac{1}{2}r(\frac{\delta}{2})$ for all $q\geq N_1$.
For each $q$, we have  
\begin{align}\label{almost dense-2-(1)}
f^{q}\left(W^u(x,\delta)\right)\supset f^{q}\left(W^u(x_0,\frac{\delta}{2})\right) \supset W^u\left(f^q(x_0), \frac{\delta}{2}\right). 
\end{align} 
 By Lemma \ref{almost dense-1}, 
 there exists an integer $N_2=N_2(\delta)$, such that 
  for any $z\in M$ with $d(z,p)<\frac{1}{2}r(\frac{\delta}{2})$, $y\in M$ and $n\geq N_2$, 
  \begin{align*}
  f^{n}\left( W^u\left(z, \frac{\delta}{2}\right) \right)\cap W^s\left(y,\frac{\delta}{2} \right)\neq \varnothing. 
  \end{align*}
 Combining this with the inequality \eqref{almost dense-2-(1)}, 
  we conclude that for any $q\geq N_1$ and $n\geq N_2$, 
\begin{align*}
f^{n+q}\left(W^u(x,\delta)\right)\cap W^s(y,\delta) \supset f^{n}\left( W^u\left(f^q(x_0), \frac{\delta}{2}\right) \right)\cap W^s\left(y,\frac{\delta}{2}\right)  \neq \varnothing. 
\end{align*}
Setting $N_0=N_1+N_2$, we complete the proof of this lemma. 
\end{proof}

    We are now ready to prove Theorem \ref{Almost Anosov-existence F-es}. 

\begin{proof}[Proof of Theorem \ref{Almost Anosov-existence F-es}]
By definition of the dynamical system $(f,M)$, 
 the weak unstable foliation is an expanding foliation, 
and conditions {\bf (C1)} and {\bf (C2)} are satisfied. 
Furthermore, 
since $f$ is uniformly contract along the stable foliation, 
Lemma \ref{almost dense-2} implies that condition {\bf (C3)} holds as well. 
Consequently, this theorem follows from 
Theorem \ref{F-pressure-2-corollary}. 
\end{proof}

\end{document}